\documentclass[12pt, reqno]{amsart}
\usepackage{amsmath, amsthm, amscd, amsfonts, amssymb, graphicx, color}
\usepackage[bookmarksnumbered, colorlinks, plainpages]{hyperref}
\hypersetup{colorlinks=true,linkcolor=red, anchorcolor=green, citecolor=cyan, urlcolor=red, filecolor=magenta, pdftoolbar=true}

\newtheorem{theorem}{Theorem}[section]
\newtheorem{lemma}[theorem]{Lemma}
\newtheorem{proposition}[theorem]{Proposition}
\newtheorem{corollary}[theorem]{Corollary}
\theoremstyle{definition}

\newtheorem{example}[theorem]{Example}

\theoremstyle{remark}
\newtheorem{remark}[theorem]{Remark}
\numberwithin{equation}{section}

\def\tcalta{\mathcal{T}(\theta, \alpha)}
\def\ata{A^{\theta,\alpha}}
\def\aat{A^{\alpha,\theta}}
\def\nk0t{\|\tilde k_0^\theta\|^{-2}}

\begin{document}
\setcounter{page}{1}

\title[Asymmetric truncated Toeplitz operators]{Characterizations of asymmetric truncated Toeplitz operators}

\author[C. C\^amara, J. Jurasik, K. Kli\'s-Garlicka, \MakeLowercase{and} M. Ptak]{Crisina C\^amara,$^1$ Joanna Jurasik,$^2$ Kamila Kli\'s-Garlicka,$^3$\\ \MakeLowercase{and} Marek Ptak $^3$ $^{*}$}

\address{$^{1}$ Center for Mathematical Analysis, Geometry and Dynamical Systems, Mathematics
Department, Instituto Superior T´ecnico, Universidade de Lisboa, Av. Rovisco Pais, 1049-
001 Lisboa, Portugal.}\email{\textcolor[rgb]{0.00,0.00,0.84}{ccamara@math.ist.utl.pt}}

\address{$^{2}$ Department of Mathematics
Maria Curie-Sk\l odowska University
20-031 Lublin, Poland}
\email{\textcolor[rgb]{0.00,0.00,0.84}{Joanna.Blicharz@live.umcs.edu.pl}}

\address{$^{3}$ Department of Applied Mathematics,
University of Agriculture, ul. Balicka 253c,\\ 30-198 Krak\'ow, Poland.}
\email{\textcolor[rgb]{0.00,0.00,0.84}{rmklis@cyfronet.pl}}

\address{$^{*}$Institute of Mathematics, Pedagogical University, ul. Podchor\c a\.zych 2, \\30-084 Krak\'ow, Poland.}
\let\thefootnote\relax\footnote{Copyright 2016 by the Tusi Mathematical Research Group.}
\email{\textcolor[rgb]{0.00,0.00,0.84}{rmptak@cyfronet.pl}}

\subjclass[2010]{ Primary: 47B35; Secondary: 30H10, 47A15.}
\keywords{Model space, truncated Toeplitz operator, kernel functions, conjugation.}

\date{Received: xxxxxx; Revised: yyyyyy; Accepted: zzzzzz.
\newline \indent $^{*}$Corresponding author}


\begin{abstract}
The aim of this paper is to investigate asymmetric truncated Toeplitz operators with $L^2$ symbols between two different model spaces given by inner functions such that one divides the other. The class of symbols corresponding to the zero operator is described. Asymmetric truncated Toeplitz operators are characterized in terms of operators of rank at most two, and the relations with the corresponding symbols are studied.

\end{abstract}

\maketitle

\section{Introduction}


Toeplitz operators on the Hardy space $H^2$, which are compositions of  a multiplication operator and the orthogonal projection from $L^2$ onto $H^2$, constitute a classical topic in operator theory. In his important paper (\cite{Sarason}) Sarason explored truncated Toeplitz operators, thus generating huge interest in this class of operators; see, for example
\cite{BCT, CMP, CCJP1, CGP, CT, CRW, GMR, GRW}.
Instead of the Hardy space $H^2$, they act on a model space $K^2_\theta=H^2\ominus\theta H^2$ associated with a given nonconstant inner function $\theta$ and a multiplication operator is composed with the orthogonal projection from $H^2$ onto $K^2_\theta$.

This work was inspired, on the one hand, by the work of Sarason, 
and, on the other hand, by \cite{CCJP}, where asymmetric  truncated Toeplitz operators  were introduced (in the context of the Hardy space $H^p$ of the half-plane, with $1<p<\infty$) and studied in the case of bounded symbols. Asymmetric  truncated Toeplitz operators involve the composition of a multiplication operator with two projections from $H^2$ onto a model space, associated with (possibly different) nonconstant inner functions $\alpha$ and $\theta$. They are a natural generalization of rectangular Toeplitz matrices, which appear in various contexts, such as in the study of finite-time convolution equations, signal processing, control theory, probability, approximation theory, diffraction problems (see for instance \cite{AC, TB, BCT, Gut, HRost, Rad,  Speck}).  

Here we consider bounded asymmetric truncated Toeplitz operators with $L^2$ symbols, defined between two model spaces
 $K^2_\theta$ and $K^2_\alpha$, where $\alpha$ divides $\theta$ ($\alpha\leqslant \theta$). We study various properties of these operators and their relations with the corresponding symbols, and we obtain a necessary and sufficient condition for a bounded operator between two model spaces to be an asymmetric truncated Toeplitz operator in terms of rank two operators, thus generalizing a corresponding result by Sarason for the case where $\alpha=\theta$. In the asymmetric case, however, a more complex connection between the operators and their symbols is revealed, which is not apparent when the two model spaces involved are the same.

The paper is organized as follows.
In Section 2 we present some auxiliary results on model spaces, their direct sum decompositions, the associated projections and conjugations. In Section 3 asymmetric truncated Toeplitz operators with $L^2$ symbols are introduced and some of their basic properties are presented. In Section 4
the class of all possible symbols for a given asymmetric truncated Toeplitz operator is described, and a necessary and sufficient condition for the operator to be zero is obtained. In Sections 5 and 6 we generalize Sarason's characterization of truncated Toeplitz operators in terms of operators of rank two at most, highlighting the similarities and differences between the symmetric and the asymmetric cases and, in Section 7,  we obtain conditions under which that characterization can be done in terms of a rank one operator. In Section 8 we address the inverse problem of determining a symbol for a given asymmetric truncated Toeplitz operator, in terms of the action of $A$ and its adjoint on certain reproducing kernel functions, as well as in terms of the characterizations presented in Sections 5 and 6.

 \section{Model spaces, conjugations and decompositions}

Let $L^2$ denote the space $L^2(\mathbb{T},m)$, where $\mathbb{T}$ is the unit circle and $m$ is the normalized Lebesgue measure on $\mathbb{T}$, and let $H^2$ be the Hardy space on the unit disc $\mathbb{D}$, identified as usual with a subspace of $L^2$.  Similarly,  $L^\infty=L^\infty(\mathbb{T},m)$ and we denote by $H^\infty$ the space of all analytic and bounded functions on $\mathbb{D}$.  Denoting by $H_0^2$ the subspace consisting of all functions in $H^2$ which vanish at $0$, we have $L^2\ominus H^2=\overline{H_0^2}$, and we denote by $P$ and $P^-$ the orthogonal projections from $L^2$ onto $H^2$ and $\overline{H_0^2}$, respectively. 


 With any given  inner function $\theta$ we associate the so called {\it model space} $K^2_\theta$, defined by $K^2_\theta=H^2\ominus \theta H^2$. We also have $K^2_\theta= H^2\cap\theta\, \overline{H_0^2}$, and thus
\begin{equation}\label{A}
f\in K^2_\theta\quad\text{if and only if}\quad \bar\theta f\in\overline{H_0^2}\quad\text{and}\ f\in H^2.
\end{equation}
In particular, if $f\in K^2_\theta$ then $\theta \bar f\in H^2_0$.
Let $P_\theta$ be the orthogonal projection $P_\theta\colon L^2\to K^2_\theta$.
We have the following:

\begin{proposition}\label{proj} Let $\theta$ be a nonconstant inner function and let $K^2_\theta=H^2\ominus \theta H^2$ be the associated model space.  Then
\begin{enumerate}
\item $P_{\theta}=\theta P^-\, \bar{\theta}P=\theta P^-\, \bar{\theta}-P^-$,
\item $P_{\theta}f=\theta P^-\, \bar{\theta}f=f-\theta P\bar\theta f$ for all $f \in H^2$,
\item $P_\theta \bar f=P_\theta P \bar f=\overline{f(0)}P_\theta1=\overline{f(0)}(1-\overline{\theta(0)}\theta)$ for all $f \in H^2$.
\end{enumerate}
\end{proposition}

Model spaces are also equipped with \textit{conjugations} (antilinear isometric involutions), which are an important tool in the study of model spaces and truncated Toeplitz operators (see for example \cite{GP, GP2, KKMP}).
For a given inner function $\theta$,  the conjugation $C_\theta$ is defined by $C_\theta\colon L^2\to L^2$,  $$C_\theta f(z)=\theta \overline{zf(z)}.$$
It is worth noting that $C_\theta$ preserves  the space $K_\theta^2$ and maps $\theta H^2$ onto $L^2\ominus{H}^2$.

Recall that for $\lambda\in\mathbb{D}$ the {\it kernel function} in $H^2$ denoted by $k_\lambda$ is given by $k_\lambda(z)=\frac{1}{1-\bar\lambda z}$. Similarly, for an inner function $\theta$, in $K^2_\theta$ the kernel function $k_\lambda^\theta$ is given by $k_\lambda^\theta=P_\theta k_\lambda$, i.e., $k_\lambda^\theta=k_\lambda(1-\overline{\theta(\lambda)}\,\theta)$. The set $\{k_\lambda^\theta: \lambda\in\mathbb D\}$ is linearly dense in $K^2_\theta$. Since $k_\lambda^\theta\in K^\infty_\theta$, where $K^\infty_\theta$ denotes the subspace $H^\infty\cap K^2_\theta$, the space $K^\infty_\theta$ is dense in $K^2_\theta$ (see \cite
{Sarason}).

Defining $\tilde{k}_\lambda^\theta=C_\theta k_\lambda^\theta$, we have in particular
\[
k_0^\theta(z)=1-\overline{\theta(0)}\theta(z) \,,\;\;\tilde{k}_0^\theta(z)=\bar z(\theta(z)-\theta(0)).
\]
It is easy to see that, for all $f\in K_\theta^2$,
\begin{equation}\label{k}
\langle f,k_0^\theta\rangle=f(0)\,,\;\;\langle f,\tilde k_0^\theta\rangle=\overline{(C_\theta f)(0)}.
\end{equation}


Now let us consider two nonconstant inner functions $\alpha$ and $\theta$. If $\bar\alpha \theta$ is an inner function, we say that $\alpha$ {\it divides} $\theta$ and we write $\alpha \leqslant \theta$.

\begin{proposition}\label{cos}
Let $\alpha, \theta$ be nonconstant inner functions such that $\alpha \leqslant \theta$. The following holds:
\begin{enumerate}
\item $K^2_\theta=K^2_\alpha\oplus\alpha K^2_{\frac{\theta}{\alpha}}$,
\item $P_\theta=P_\alpha +\alpha P_{\frac{\theta}{\alpha}}\bar\alpha$,
\item  $ k_0^\theta= k_0^\alpha+\overline{\alpha(0)}\alpha k_0^{\frac{\theta}{\alpha}}$,

\item  $\tilde k_0^\theta=\frac{\theta}{\alpha}(0)\tilde k_0^\alpha+\alpha\tilde k_0^{\frac{\theta}{\alpha}}$,
\item $P_\alpha k_0^\theta=k_0^\alpha\,,\,P_\alpha \tilde k_0^\theta=\frac{\theta}{\alpha}(0) \tilde k_0^\alpha$.

\end{enumerate}
\end{proposition}

\begin{proof} (1),\ (2) and (3) were proved in \cite[p. 97]{GMR}. We only need to   prove (4), and (5) follows immediately from (3) and (4). In fact we have
\[
\tilde{k}_0^{\theta}=\bar z(\theta-\theta(0))=\alpha\bar z\tfrac{\theta}{\alpha}-\bar z\alpha(0)\tfrac{\theta}{\alpha}(0)= \alpha\bar z (\tfrac{\theta}{\alpha}-\tfrac{\theta}{\alpha}(0))+\tfrac{\theta}{\alpha}(0)\bar z(\alpha-\alpha(0)),
\]
where $\bar z (\frac{\theta}{\alpha}-\frac{\theta}{\alpha}(0))=\tilde k_0^{\frac{\theta}{\alpha}}\,,\,\bar z(\alpha-\alpha(0))=\tilde k_0^\alpha$.
\end{proof}

The following proposition describes some relations between decompositions and conjugations. Note that, if $\alpha \leqslant \theta$, any $f\in K_\theta^2$ can be uniquely decomposed as $f=f_1+\alpha f_2$ for some $f_1\in K^2_\alpha$ and some $f_2\in  K^2_{\frac{\theta}{\alpha}} $, or as $f=f_2+\frac{\theta}{\alpha} f_1$, for some $f_1\in K^2_\alpha$ and some $f_2\in  K^2_{\frac{\theta}{\alpha}} $. Then the  conjugation $C_{\theta}$ can be seen as $C_{\theta} \colon K^2_{\theta} = K^2_{\alpha}\oplus \alpha K^2_{\frac{\theta}{\alpha}} \to K^2_{\theta} = K^2_{\frac{\theta}{\alpha}}\oplus \frac{\theta}{\alpha} K^2_{\alpha}$, or as  $C_{\theta} \colon K^2_{\theta} = K^2_{\frac{\theta}{\alpha}}\oplus \frac{\theta}{\alpha} K^2_{\alpha} \to K^2_{\theta} = K^2_{\alpha}\oplus \alpha K^2_{\frac{\theta}{\alpha}}$. Now we have:

\begin{proposition}\label{rozklad}
Let $\alpha, \theta$ be nonconstant inner functions such that $\alpha \leqslant \theta$. Then, if $f_1\in K^2_\alpha$ and $f_2\in  K^2_{\frac{\theta}{\alpha}} $,

\begin{enumerate}
\item $C_{\theta}(f_1+\alpha f_2)= C_{\frac{\theta}{\alpha}}f_2+\tfrac{\theta}{\alpha}C_{\alpha}f_1$,
\item $C_{\theta}(f_2+ \tfrac{\theta}{\alpha} f_1)=C_{\alpha}f_1+\alpha C_{\frac{\theta}{\alpha}}f_2$.

\end{enumerate}
\end{proposition}

\begin{proof}
Let $f_1\in K^2_{\alpha}$, $f_2\in K^2_{\frac{\theta}{\alpha}}$. Then $C_{\theta}(f_1+\alpha f_2)= \theta \bar{z} \bar{f_1}+\theta \bar{z}\bar{\alpha} \bar f_2=\tfrac{\theta}{\alpha}\alpha \bar{z}\bar{f_1}+\tfrac{\theta}{\alpha}\bar{z}\bar{f_2}=
C_{\frac{\theta}{\alpha}}(f_2)+\tfrac{\theta}{\alpha}C_{\alpha}(f_1)$ and $C_{\theta}(f_2+\frac{\theta}{\alpha} f_1)= \theta \bar{z} \bar{f_2}+\theta \bar{z}\frac{\bar{\theta}}{\bar{\alpha}}\bar{f_1}=\alpha \bar{z}\bar{f_1}+\alpha \frac{\theta}{\alpha}\bar{z}\bar f_2=
C_{\alpha}(f_1)+\alpha C_{\frac{\theta}{\alpha}}(f_2)$.
\end{proof}

\begin{corollary}\label{wn2} Let $\alpha, \theta$ be nonconstant inner functions such that $\alpha \leqslant \theta$. We have, for $f\in K^2_\theta$,
\begin{enumerate}
\item $P_\alpha C_\theta f=C_\alpha P_\alpha (\frac{\bar\theta}{\bar\alpha}f)$,
 \item $P_{\frac{\theta}{\alpha}}C_\theta f=C_{\frac{\theta}{\alpha}}P_{\frac{\theta}{\alpha}}\bar\alpha f$.\end{enumerate}
\end{corollary}

Let $S$ be the unilateral shift on the Hardy space $H^2$ and, for a nonconstant inner function $\theta$, let $S_\theta=P_\theta S_{|K^2_\theta}$  be the compression of  $S$ to $K^2_\theta$. The space $K^2_\theta$ is invariant for $S^*$, thus  $(S_\theta)^*=  {S^*}_{|K^2_\theta}$. Note that, for any $f\in K^2_\theta$,

\begin{equation}\label{B}
S_\theta f=zf-\overline{(C_\theta f)(0)}\,\theta=Sf-\langle f,\tilde k_0^\theta\rangle\theta,
\end{equation}
\begin{equation}\label{C}
S_\theta^* f=\bar z(f-f(0)).
\end{equation}
In particular,
\begin{equation}\label{D}
S_\theta^* k_0^\theta=-\overline{\theta(0)} \,\tilde k_0^\theta\,,\;S_\theta\tilde k_0^\theta=-\theta(0)k_0^\theta.
\end{equation}

The function $k_0^\theta$  is a cyclic vector for $S_\theta$ and $\tilde{k}_0^\theta$ is a cyclic vector for $S_\theta^*$ (see \cite[Lemma 2.3]{Sarason}).
 In what follows we will use the {\it defect operators} $I_{K^2_\theta}-S_\theta S^*_\theta=k_0^\theta\otimes k_0^\theta$ and $I_{K^2_\theta}-S^*_\theta S_\theta=\tilde k_0^\theta\otimes \tilde k_0^\theta$, using the notation $(x\otimes y)z=\langle z,y\rangle x$ for any $x,y,z$ in a Hilbert space $H$  (\cite[Lemma 2.4]{Sarason}). We list below some simple properties that will be used later.

\begin{proposition}\label{k-inf}
Let $\alpha$ and $\theta$ be nonconstant inner functions such that $\alpha\leqslant\theta$. Then
\begin{enumerate}
 \item  $P_\alpha S_\theta=S_\alpha P_\alpha$ on $K^2_\theta$,
\item $ S^*_\theta P_\alpha=S^*_\alpha $ on $K^2_\alpha$,
  \item $S^n_\theta k_0^\theta\in K^\infty_\theta$ for $n\geq 0$,
  \item $(S^n_\theta)^*\tilde k_0^\theta\in K^\infty_\theta$ for $n\geq 0$.
\end{enumerate}
\end{proposition}
\begin{proof} The proof of (1) and (2) is immediate.
By \cite[Theorem 9.2.2]{GMR} we know that $S_\theta^n f=P_\theta S^n f$ for any $f\in K_\theta^2$, $n\geq 0$. Hence
\[S_\theta^n k_0^\theta=P_\theta S^n k_0^\theta=P_\theta(z^n(1-\overline{\theta(0)}\theta))=P_\theta z^n=z^n-\theta P (\bar\theta z^n).\] Since $P (\bar\theta z^n)$ is a polynomial, we get $S^n_\theta k_0^\theta\in K^\infty_\theta$, which proves (3).

On the other hand,
$$(S^*_\theta)^n\tilde k_0^\theta=P_\theta\bar z^{n+1}(\theta-\theta(0))=\theta\bar z^{n+1}-P^-\theta\bar z^{n+1};
$$
since $P^- (\theta \bar z^n)$ is a polynomial in $\bar z$, we get (4).
\end{proof}


\section{Asymmetric truncated Toeplitz operators}

Let $\alpha$, $ \theta$ be nonconstant inner functions. For $\varphi\in L^2$ we define an operator $A_\varphi^{\theta,\alpha}\colon \mathcal{D}\subset K^2_\theta\to K^2_\alpha$,
    as
   $A_\varphi^{\theta,\alpha} f=P_{\alpha}(\varphi f)$ having domain %
   $\mathcal{D}=\mathcal{D}( A_\varphi^{\theta,\alpha})=\{ f\in K^2_\theta\colon \varphi f\in L^2\}$.
 The operator $A_\varphi^{\theta,\alpha}$ is closed and densely defined in $K^2_\theta$. Note that $K^\infty_\theta\subset \mathcal{D}(A_\varphi^{\theta,\alpha})$.
   The  operator   $A_\varphi^{\theta,\alpha}$ will be called {\it an asymmetric truncated Toeplitz operator}.
      If this operator is bounded, then it admits a unique bounded extension to $K^2_\theta$, $A_{\varphi}^{\theta,\alpha}\colon K^2_\theta\to K^2_\alpha$. By $\tcalta$ we denote the space of all bounded asymmetric
 truncated Toeplitz operators. 
    For $\alpha=\theta$ we will write $A_\varphi^{\theta}$ instead of $A_\varphi^{\theta,\theta}$ (such operators are called {\it truncated Toeplitz operators} and were studied by Sarason in \cite{Sarason}) and $\mathcal{T}(\theta)$ instead of $\mathcal{T}(\theta,\theta)$.

It is easy to see that the following holds.

\begin{proposition}\label{wl}Let $\alpha, \theta$ be nonconstant inner functions such that $\alpha\leqslant \theta$.  Then
\[
 \ata_{k_0^\alpha}=P_\alpha=\ata_{\overline{k_0^\theta}}.
 \]
If $\varphi\in L^2$, we have for all $f\in  K^\infty_\alpha$
 \[
A_\varphi^{\theta,\alpha}f=A_\varphi^{\alpha}f.
\]
\end{proposition}


\begin{proposition}\label{l3}
Let $\alpha$, $ \theta$ be any inner functions and $\varphi\in L^2$.
Then
\begin{equation}\label{e-adj}\langle A_{\varphi}^{\theta,\alpha} f, g\rangle=
\langle f,A_{\bar\varphi}^{\alpha,\theta} g \rangle
\quad\text{for all}\quad f\in \mathcal{D}(A_{\varphi}^{\theta,\alpha}),\ g\in \mathcal{D}(A_{\bar\varphi}^{\alpha,\theta}).
\end{equation}
Moreover,
$\mathcal{D}(A_{\bar\varphi}^{\alpha,\theta})=\mathcal{D}(\big(A_{\varphi}^{\theta,\alpha}\big)^*)$ and
 $\big(A_{\varphi}^{\theta, \alpha}\big)^*=A_{\bar\varphi}^{ \alpha, \theta}$.
 \end{proposition}

\begin{proof}
A straightforward calculation shows that
\begin{multline*}
\langle A_{\varphi}^{\theta,\alpha} f, g\rangle=\langle  P_\alpha (\varphi  f),g\rangle=
\langle \varphi f, g\rangle=\int f{\varphi} \bar g dm \\=
 \langle f, \bar\varphi g\rangle=\langle f, P_\theta (\bar\varphi  g)\rangle=
\langle f, A_{\bar\varphi}^{\alpha,\theta}g \rangle.
\end{multline*}Note also that $\mathcal{D}(A_{\bar\varphi}^{\alpha,\theta})=
\{ g\in K^2_\alpha\colon \bar\varphi g\in L^2\}$. So, if $g\in \mathcal{D}(A_{\bar\varphi}^{\alpha,\theta})$, there is $h\in K^2_\theta$ such that $$\langle A_{\varphi}^{\theta,\alpha} f, g\rangle=\langle f,h\rangle,$$  for all $f\in \mathcal{D}(A_{\varphi}^{\theta,\alpha})$, i.e., such that $\varphi f \in L^2$, and therefore $g\in \mathcal{D}((A_{\varphi}^{\theta,\alpha})^*)$. In fact, taking $h=P_\theta(\bar\varphi f)$, we have $\langle A_{\varphi}^{\theta,\alpha} f, g\rangle=\langle P_\alpha(\varphi f),g\rangle=\langle \varphi f,g\rangle=\langle f,\bar\varphi g\rangle=\langle f, P_\theta(\bar\varphi g)\rangle=\langle f,h\rangle$. The converse is similarly true.
\end{proof}

\begin{proposition}\label{p1}
Let $\alpha, \theta$ be nonconstant inner functions such that $\alpha\leqslant \theta$.
  Let $A_\psi^{\theta,\alpha}$ be an asymmetric truncated Toeplitz operator with $\psi\in H^2$. Then $$S_\alpha A_\psi^{\theta,\alpha}f =A_\psi^{\theta,\alpha}  S_\theta f \quad\text{ for all}\ \  f\in K^\infty_\theta.$$
\end{proposition}

\begin{proof}
Let $f\in K_\theta^\infty$. Then by Proposition \ref{proj}
\[
S_\alpha A_\psi^{\theta,\alpha} f=S_\alpha P_\alpha(\psi f)= P_\alpha z(\psi f-\alpha P(\bar{\alpha}\psi f))
=P_\alpha(z\psi f),
\]
since $P_\alpha (\alpha z P (\bar{\alpha}\psi f))=0$. On the other hand, taking \eqref {B} into account, for all $f\in K_\theta^\infty$ we have
\[
A_\psi^{\theta,\alpha} S_\theta f=P_\alpha\psi  S_\theta f= P_\alpha\psi (z f- \langle f,\tilde k_0^\theta\rangle\theta)= P_\alpha(z\psi f),
\]
since $P_\alpha (\psi \theta )=0$.
\end{proof}

\begin{remark}
 Theorem 3.1.16 \cite{Bercovici} implies that for nonconstant inner functions $\alpha, \theta$ such that $\alpha\leqslant \theta$, if a bounded operator $A\colon K^2_\theta\to K^2_\alpha$ intertwines $S_\alpha$, $S_\theta$, i.e., $S_\alpha A =A S_\theta$, then $A=A_\psi^{\theta,\alpha}$ for some $\psi\in H^\infty$.
\end{remark}
\begin{example}
One can ask, whether a similar result can be obtained for $A_\psi^{\alpha, \theta}$ with $\alpha\leqslant \theta$ and $\psi\in H^2$, but the answer is negative. For example, let $\alpha=z^2$, $\theta=z^n$, $n>5$, $\psi=z^3$ and $f=z$. Then $S_\theta A_\psi^{\alpha,\theta}f=z^5$ but $A_\psi^{\alpha,\theta}S_\alpha f=0$.
\end{example}

 In the next proposition we describe the action of some (not necessarily bounded) asymmetric truncated Toeplitz operators on some particular functions. These properties will be used later on.

\begin{proposition}\label{kzero}
Let $\alpha, \theta$ be nonconstant inner functions such that $\alpha \leqslant \theta$, and let $\psi\in K^2_{\alpha}, \chi\in K^2_{\theta}$, with $\chi=\chi_{\frac{\theta}{\alpha}}+\frac{\theta}{\alpha}\chi_\alpha$ where $\chi_{\frac{\theta}{\alpha}}\in K^2_{\frac{\theta}{\alpha}}$ and $\chi_\alpha\in K^2_\alpha$.
 Then
\begin{enumerate}
\item $\ata_\psi k_0^\theta=\psi$,
\item $\ata_{\bar\chi} k_0^\theta=\overline{\chi(0)}k_0^\alpha-
    \overline{\theta(0)}\alpha(\bar\chi_\alpha-\overline{\chi_\alpha(0)}),$
  \item $A^{\theta,\alpha}_{\psi}\tilde{k}_0^{\theta}=\frac{\theta}{\alpha}(0)\psi(0) \tilde{k}_0^{\alpha} - \theta(0)S^*_{\alpha}\psi
=\bar z(\frac{\theta}{\alpha}(0)\psi(0)\alpha-\theta(0)\psi)$,
      \item $A^{\theta,\alpha}_{\bar{\chi}}\tilde{k}_0^{\theta}=P_{\alpha}C_{\theta}\chi=C_\alpha P_\alpha(\frac{\bar\theta}{\bar\alpha}\chi)=C_\alpha \chi_\alpha,$
  \item  $A^{\alpha,\theta}_{\bar{\psi}}k_0^{\alpha}= \overline{\psi(0)}k_0^\theta-\overline{\alpha(0)}(\alpha\bar\psi-\overline{\frac{\theta}{\alpha}(0)\psi(0)}\,\theta)$,
  \item $A^{\alpha,\theta}_\chi k_0^\alpha=\chi-\overline{\alpha(0)}\alpha\chi_{\frac{\theta}{\alpha}}$,
  \item $A^{\alpha,\theta}_{\bar{\psi}}\tilde{k}_0^{\alpha}=C_{\alpha}\psi$,
  \item $A^{\alpha,\theta}_\chi \tilde k_0^\alpha=P_\theta C_\theta(\bar\chi\frac{\theta}{\alpha})-\alpha(0)S^*_\theta\chi.$
  \end{enumerate}
\end{proposition}

\begin{proof}
 (1) follows from
$$\ata_\psi k_0^\theta=P_\alpha(\psi-\overline{\theta(0)}\theta\psi)=\psi.$$
To prove (2) we calculate firstly
\[
P_\alpha (\theta\bar\chi)=P_\alpha(\alpha \tfrac{\theta}{\alpha}{\bar\chi_{\frac{\theta}{\alpha}}})+P_\alpha(\alpha{\bar\chi_\alpha})
=P_\alpha(\alpha{\bar\chi_\alpha})=\alpha(\bar\chi_\alpha-\overline{\chi_\alpha(0)}).
\]
 Hence by Proposition \ref{proj} (3)
$$\ata_{\bar\chi} k_0^\theta=P_\alpha\bar\chi-\overline{\theta(0)}P_\alpha(\theta\bar\chi)=\overline{\chi(0)}k_0^\alpha-
  \overline{\theta(0)}  \alpha(\bar\chi_\alpha-\overline{\chi_\alpha(0)}).$$
To prove (3) note that by Proposition \ref{cos}
\begin{multline*}
A^{\theta,\alpha}_{\psi}\tilde{k}_0^{\theta}=
P_{\alpha}(\psi P_{\theta}\tilde{k}_0^{\theta})=
P_{\alpha}(\psi P_{\alpha}\tilde{k}_0^{\theta})+P_{\alpha}(\psi \alpha P_{\frac{\theta}{\alpha}}\bar\alpha\tilde{k}_0^{\theta})=
P_{\alpha}(\psi P_{\alpha}\tilde{k}_0^{\theta}) \\=
P_{\alpha}(\psi \tfrac{\theta}{\alpha}(0)\tilde{k}_0^{\alpha})=
\tfrac{\theta}{\alpha}(0)P_{\alpha}(\psi \bar{z} \alpha) - \theta(0) P_{\alpha} (\psi\bar{z}) =
\tfrac{\theta}{\alpha}(0) \psi(0)\tilde{k}_0^{\alpha}-\theta(0)S^*_{\alpha}\psi,
\end{multline*}
since
\begin{multline*}
P_{\alpha} \psi\bar{z}\alpha=(\alpha P^-\bar{\alpha}-P^-)(\psi \alpha \bar{z})=\alpha P^-(\psi \bar{z})-\psi(0)\alpha(0) \bar{z}\\
 =\alpha \psi(0)\bar{z}-\psi(0)\alpha(0)\bar{z}=\psi(0)\tilde{k}_0^{\alpha}.
\end{multline*}
To show (4) we calculate
\begin{equation*}
A^{\theta,\alpha}_{\bar{\chi}}\tilde{k}_0^{\theta}=
P_{\alpha}(\bar{\chi}\bar{z}(\theta-\theta(0)))=
P_{\alpha}(\theta\bar{\chi}\bar{z})-\theta(0)P_{\alpha}(\bar{\chi}\bar{z}) =
P_{\alpha}C_{\theta}\chi,
\end{equation*}
and by Corollary \ref{wn2}
$$P_{\alpha}(C_{\theta}\chi)=C_{\alpha}P_{\alpha} \tfrac{\bar{\theta}}{\bar{\alpha}} \chi=C_{\alpha}\chi_{\alpha}.$$
The equality in (5) follows from
$$A^{\alpha,\theta}_{\bar{\psi}}k_0^{\alpha}= P_\theta(\bar\psi(1-\overline{\alpha(0)}\alpha))=
\overline{\psi(0)}k_0^\theta-\overline{\alpha(0)}P_\theta(\alpha\bar\psi)
$$and
$$P_\theta(\alpha\bar\psi)= \alpha\bar\psi-\theta P(\bar\theta\alpha\bar\psi)=\alpha\bar\psi-\theta P{\tfrac{\bar\theta}{\bar\alpha}\bar\psi}=\alpha\bar\psi-\theta\, \overline{\tfrac{\theta}{\alpha}(0)\,\psi(0)}.$$
The equality in (6) follows from
$$A^{\alpha,\theta}_\chi k_0^\alpha=P_\theta(\chi(1-\overline{\alpha(0)}\alpha))=\chi-\overline{\alpha(0)}P_\theta(\alpha\chi)=\chi-\overline{\alpha(0)}\alpha\chi_{\frac{\theta}{\alpha}},$$
and (7) follows from
\begin{equation*}
A^{\alpha,\theta}_{\bar{\psi}}\tilde{k}_0^{\alpha}=
P_{\theta}(\bar{\psi}\bar{z}(\alpha-\alpha(0)))=
P_{\theta}(\alpha \bar{z}\bar{\psi})-\alpha(0) P_{\theta}(\bar{z}\bar{\psi})=P_{\theta}C_{\alpha}\psi=C_\alpha\psi.
\end{equation*}
Finally (8) follows from
\begin{equation*}
A^{\alpha,\theta}_\chi \tilde k_0^\alpha=P_\theta(\chi \tilde k_0^\alpha)=P_\theta(\alpha\chi \bar z)-\alpha(0)P_\theta(\bar z\chi)=P_\theta C_\theta(\bar\chi\tfrac{\theta}{\alpha})-\alpha(0)S^*_\theta\chi.
\end{equation*}

\end{proof}


\section{The symbols of zero in $\mathcal{T}(\theta,\alpha)$}
In what follows, for $f\in H^2$ let $f_i$, $f_o$  denote the inner and outer factors, respectively, in an inner-outer factorization $f=f_i\;f_o$. Recall that $f_i$ and $f_o$  are defined up to multiplication by a constant.
 For $\alpha$ and $\theta$  nonconstant inner functions, let $GCD(\theta,\alpha)$ denote the greatest common divisor of $\theta$ and $\alpha$, which is also defined up to a constant.

In this section we study the symbols for which the corresponding asymmetric truncated Toeplitz operator  is  the zero operator. This is equivalent to the kernel of the operator being equal to $K^2_\theta$.

We start by characterizing the kernels of some asymmetric truncated Toeplitz operators. The following result generalizes the disk versions of Theorem 7.2 in \cite{CCJP} and Theorem 3.2 in \cite{CCJP1}.


\begin{theorem}\label{l1} Let $\alpha$ and $\theta$ be nonconstant inner functions and let $\varphi\in H^\infty$, $\varphi\ne 0$, with inner factor
$\varphi_i$.
Then
\begin{enumerate}
\item $\ker A_\varphi^{\theta,\alpha}=K^2_\theta\cap\frac{\alpha}{\psi}H^2$ with $\psi=GCD(\alpha,\varphi_i)$; in particular if $\frac{\alpha}{\psi}\leqslant \theta$,  then
 $\ker A_\varphi^{\theta,\alpha}=\frac{\alpha}{\psi}K^2_{\frac{\theta \psi}{\alpha}}$\,;
 \item if $\theta\leqslant\alpha$, then $\ker A_{\bar\varphi}^{\theta,\alpha}=\ker \big(A_\varphi^{\alpha,\theta}\big)^*=K^2_{GCD(\theta,\psi)}$,  with $\psi=GCD(\alpha,\varphi_i)$.
 \end{enumerate}
\end{theorem}

\begin{proof}
(1) Let $f\in K^2_\theta$.
Note that $f\in \ker A_\varphi^{\theta,\alpha}$ if and only if $\varphi f\in\alpha H^2$, which is equivalent to $\alpha$ dividing $\varphi_i f_i$, where $f_i$ is the inner factor of $f$. Since $\frac{\alpha}{\psi}$ and $\frac{\varphi_i}{\psi}$ are relatively prime, $f_i$ is divisible  by $\frac{\alpha}{\psi}$. Hence $\ker A_\varphi^{\theta,\alpha}=\{f\in K^2_\theta : f\in \frac{\alpha}{\psi}H^2 \}= K^2_\theta\cap\frac{\alpha}{\psi}H^2$. If $\frac{\alpha}{\psi}\leqslant \theta$, the result follows from the decomposition $K^2_\theta=K^2_{\frac{\alpha}{\psi}}\oplus\frac{\alpha}{\psi}K^2_{\frac{\theta \psi}{\alpha}}$.

(2) If $\theta\leqslant\alpha$, in which case $K^2_\theta\subset K^2_\alpha$, we have
\begin{align*}\ker A_{\bar\varphi}^{\theta,\alpha}&=\{f\in K^2_\theta\colon P_\alpha (\bar{\varphi}f)=0\}=
\{f\in K^2_\theta\colon P (\bar{\varphi}f)=0\}\\&=\{f\in K^2_\theta\colon \bar\varphi f\in \overline{H^2_0}\}=\{f\in K^2_\theta\colon \varphi\alpha\bar z\bar f\in \alpha H^2\}\\
&=\{f\in K^2_\theta\colon \varphi \,C_\alpha f\in\alpha H^2\}=\{f\in K^2_\theta\colon C_\alpha f\in \ker A_\varphi^\alpha=K^2_\alpha\cap\tfrac{\alpha}{\psi} H^2\}\end{align*}
by (1).  Now,
$ C_\alpha f\in K^2_\alpha\cap\frac{\alpha}{\psi} H^2$ if and only if $ f\in  K^2_\alpha\cap\psi\,\overline{H^2_0}$,  which is equivalent to $f\in K^2_\alpha\cap K^2_\psi=K^2_\psi.$ Hence $f\in K^2_\theta\cap K^2_\psi= K^2_{GCD(\theta,\psi)}$.
\end{proof}

\begin{lemma}\label{l2} Let $\alpha$ and $\theta$ be nonconstant inner functions   and let $\varphi\in H^2$.
Assume that $A_\varphi^{\alpha,\theta}$ is an asymmetric truncated Toeplitz operator. Then $A_\varphi^{\alpha,\theta}=0$ if and only if $ \varphi\in\theta H^2$.
\end{lemma}

\begin{proof} If $ \varphi\in\theta H^2$ and $f\in K^\infty_\alpha$, then $A_\varphi^{\alpha,\theta}f=P_\theta(\varphi f)=P_\theta (\theta \varphi_1 f)=0$, where $\varphi=\theta\varphi_1$ and $\varphi_1\in H^2$.  Since $  K^\infty_\alpha$ is dense in $K^2_\alpha$, then $A_\varphi^{\alpha,\theta}=0$.

For the converse implication
let $\psi=GCD(\theta,\varphi_i)$.
If $A_\varphi^{\alpha,\theta}=0$ and $f\in  K^\infty_\alpha$, then $\varphi f\in \theta H^2$. Hence $\theta$ divides $\varphi_i f_i$ and consequently $\frac{\varphi_i}{\psi}f_i\in \frac{\theta}{\psi}H^2$. Since $\frac{\varphi_i}{\psi}$ and $\frac{\theta}{\psi}$ are relatively prime, then $f_i\in \frac{\theta}{\psi}H^2$. By density of $K^\infty_\alpha$  in $K^2_\alpha$ we have   $K^2_\alpha\subset \frac{\theta}{\psi}H^2$, which is possible only if $\frac{\theta}{\psi}$ is a constant, because $K^2_\alpha$ is the kernel of the Toeplitz operator $T_{\bar\alpha}$ (\cite[Proposition 2.4]{CMP}). Indeed, kernels of Toeplitz operators are near invariant with respect to functions from $\overline{H^\infty}$ (i.e., if $g\in \ker T_{\bar\alpha}$ and $h\in H^\infty$, then $\bar hg\in H^2$ implies $\bar hg\in\ker T_{\bar\alpha}$). Hence $K^2_\alpha\subset \frac{\theta}{\psi}H^2$ implies that, for every $g_0\in K^2_\alpha$, we have $\overline{\frac{\theta}{\psi}}g_0=g_1\in H^2$, so that $g_1\in K^2_\alpha$. Repeating this reasoning for $g_1$ we conclude that $g_0$ can be indefinitely divided by $\frac{\theta}{\psi}$, which is possible only if $\frac{\theta}{\psi}$ is a constant.
\end{proof}

\begin{corollary}
If $\varphi\in H^\infty$, then $A_{\varphi}^{\theta,\alpha}=0$ if and only if $\varphi\in\alpha H^\infty$ and $A_{\bar\varphi}^{\theta,\alpha}=0$ if and only if $\varphi\in\theta H^\infty$.
\end{corollary}

\begin{proof}
If $\ata_\varphi=0$, then by the previous lemma $\bar\alpha\varphi\in H^2\cap L^\infty=H^\infty$. Therefore $\varphi\in \alpha H^\infty$. The converse is obvious.
\end{proof}
The next theorem establishes a necessary and sufficient condition for a bounded asymmetric truncated Toeplitz operator to be the zero operator  in terms of its symbol.


\begin{theorem}\label{zero}
Let $\alpha$, $\theta$ be nonconstant inner functions such that $\alpha\leqslant \theta$. Let $A_\varphi^{\theta,\alpha}\colon K^2_\theta \to K^2_\alpha$ be a bounded asymmetric truncated Toeplitz operator with $\varphi\in L^2$. Then
$A_\varphi^{\theta,\alpha}=0$ if and only if $\varphi\in \alpha H^2 +\overline{\theta H^2}$.
\end{theorem}

\begin{proof}
Let us assume that $\varphi=\alpha h_1+\bar{\theta}\bar{h_2}$ for some $h_1, h_2\in H^2$. Take $f\in K^\infty_\theta$. Then
$\varphi f= \alpha h_1 f+ \bar{\theta} \bar{h_2}f$ and $P_\alpha(\alpha h_1 f)=0$.
Since $\bar{\theta}f\in \bar{\theta} K^2_\theta \subset \overline{z H^2}$, then $\bar{\theta}f \bar h_2\in \overline{z H^2}$ and $P_\alpha (\bar{\theta}f \bar h_2)=0$.
Hence $A_\varphi^{\theta,\alpha}=0$ on a dense subset of $K^2_\theta$, which implies that $A_\varphi^{\theta,\alpha}=0$.

For the converse implication let us take $\varphi =\psi+\bar{\chi}\in L^2$ with $\psi, \chi \in H^2$ such that $A_\varphi^{\theta,\alpha}=0$. Note that
$A_\psi^{\theta,\alpha}=-A_{\bar \chi}^{\theta,\alpha}$. Thus $A_\psi^\alpha\,f=A^{\theta,\alpha}_\psi\,f=-A^{\theta,\alpha}_{\bar\chi}\,f=-A^{\alpha}_{\bar\chi}f$ for all $f\in K^\infty_\alpha$ by Proposition \ref{wl}. Hence $A_\psi^\alpha$ commutes  on $K^\infty_\alpha$ not only with $S_\alpha$ but also with $S^*_\alpha$, see \cite{Sarason} and Proposition \ref{p1}. Therefore, since $k_0^\alpha\in K^\infty_\alpha$, $S^*_\alpha k_0^\alpha=-\overline{\alpha(0)}\tilde k_0^\alpha\in K^\infty_\alpha$ (cf. \eqref{D})
and
$(I_{K^2_\alpha}-S_\alpha S^*_\alpha)k_0^\alpha  \in K^\infty_\alpha$, it follows that
\begin{multline*}
A_\psi^{\theta,\alpha}
 (I_{K^2_\alpha}-S_\alpha S^*_\alpha)k_0^\alpha=
(A_\psi^\alpha - S_\alpha A_\psi^\alpha S^*_\alpha) k_0^\alpha=
( A_\psi^\alpha - S_\alpha S^*_\alpha A_\psi^\alpha ) k_0^\alpha=\\= (I_{K^2_\alpha}-S_\alpha S_\alpha^*)A_\psi^\alpha  k_0^\alpha=
 (k_0^\alpha \otimes k_0^\alpha)A_\psi^\alpha k_0^\alpha=\langle A_\psi^\alpha k_0^\alpha, k_0^\alpha\rangle k_0^\alpha.
\end{multline*}
On the other hand, by Lemma 2.4 \cite{Sarason}
\begin{multline*}
A_\psi^{\theta,\alpha}
(I_{K^2_\alpha}-S_\alpha S^*_\alpha)k_0^\alpha=A_\psi^\alpha (k_0^\alpha\otimes k_0^\alpha)k_0^\alpha=\\
((A_\psi^\alpha k_0^\alpha)\otimes k_0^\alpha)k_0^\alpha=\langle k_0^\alpha, k_0^\alpha\rangle A_\psi^\alpha k_0^\alpha.
\end{multline*}
Hence
$$\langle k_0^\alpha, k_0^\alpha\rangle A_\psi^\alpha k_0^\alpha=\langle A_\psi^\alpha k_0^\alpha, k_0^\alpha\rangle k_0^\alpha, $$
and it follows that there is $c\in\mathbb{C}$ such that
\begin{equation}\label{e1} A_\psi^\alpha k_0^\alpha=c k_0^\alpha.\end{equation}
Therefore
$$0=(A_\psi^\alpha -cI_{K^2_\alpha})k_0^\alpha =P_\alpha ((\psi -c)(1-\overline{\alpha(0)}\alpha))=P_\alpha (\psi -c),$$
which implies that $\psi -c \in\alpha H^2$.

Let us now consider $A^{\theta,\alpha}_{\bar\chi}$. We have $A^{\theta,\alpha}_{\bar\chi}=-A^{\theta,\alpha}_{\psi}$ and $A^{\theta,\alpha}_{\psi}$ intertwines $S_\alpha$ and  $S_\theta$ on $K^\infty_\theta$ by Proposition \ref{p1}. Thus
$A^{\theta,\alpha}_{\bar\chi}S_\theta f=S_\alpha A^{\theta,\alpha}_{\bar\chi} f$ for $f\in K^\infty_\theta$. Also $A^{\theta,\alpha}_{\bar\chi}S^*_\theta P_\alpha=A_{\bar\chi}^\alpha S^*_\alpha P_\alpha$, since $K_\alpha^2$ is invariant for both $S^*_\theta$ and  $A^{\theta,\alpha}_{\bar\chi}$.
Since $P_\alpha k_0^\theta=k_0^\alpha\in K^\infty_\alpha\subset K^\infty_\theta$ and as above $S^*_\theta P_\alpha k_0^\theta\in K^\infty_\theta$, hence
\begin{align*}
A&_{\bar\chi}^{\theta,\alpha} (I_{K^2_\theta}-S_\theta S^*_\theta)P_\alpha k_0^\theta=
(A_{\bar\chi}^{\theta,\alpha} P_\alpha-S_\alpha A_{\bar\chi}^{\theta,\alpha} S^*_\theta P_\alpha) k_0^\theta=\\&=
(A_{\bar\chi}^{\theta,\alpha} P_\alpha-S_\alpha A_{\bar\chi}^{\alpha} S^*_\alpha P_\alpha) k_0^\theta=
( A_{\bar\chi}^{\theta,\alpha} P_\alpha- S_\alpha S^*_\alpha A_{\bar\chi}^{\alpha} P_\alpha) k_0^\theta=\\&= (I_{K^2_\alpha}-S_\alpha S_\alpha^*)(A_{\bar\chi}^{\theta,\alpha})_{|K^2_\alpha} P_\alpha k_0^\theta=
 (k_0^\alpha \otimes k_0^\alpha)A_{\bar\chi}^{\theta,\alpha} k_0^\alpha=\langle A_{\bar\chi}^{\theta,\alpha} k_0^\alpha, k_0^\alpha\rangle k_0^\alpha
\end{align*}
and
\begin{multline*}
A_{\bar\chi}^{\theta,\alpha} (I_{K^2_\theta}-S_\theta S^*_\theta)P_\alpha k_0^\theta=A_{\bar\chi}^{\theta,\alpha} (k_0^\theta\otimes k_0^\theta)P_\alpha k_0^\theta=\\=
(A_{\bar\chi}^{\theta,\alpha} k_0^\theta\otimes k_0^\theta)k_0^\alpha
=\langle k_0^\alpha, k_0^\theta\rangle A_{\bar\chi}^{\theta,\alpha} k_0^\theta.
\end{multline*}
It follows that
$
A_{\bar\chi}^{\theta,\alpha} k_0^\theta = c_1 k_0^\alpha.$
 Moreover, \begin{equation}\label{q2}
 A_{\bar\chi}^{\theta,\alpha}=c_1 P_\alpha\in\tcalta .\end{equation} Indeed, since
 $A_{\bar\chi}^{\theta,\alpha}=-A_\psi^{\theta,\alpha}$, by Proposition~\ref{p1} the operator $A_{\bar\chi}^{\theta,\alpha}$ intertwines $S_\theta$ and $S_\alpha$ on $K^\infty_\theta$. By Proposition~\ref{k-inf} we know that $S^n_\theta k_0^\theta \in K^\infty_\theta$ for $n\geq 0$, thus we have
$$
A_{\bar\chi}^{\theta,\alpha}S^n_\theta k_0^\theta = S^n_\alpha A_{\bar\chi}^{\theta,\alpha} k_0^\theta=c_1 S_\alpha^n P_\alpha k_0^\theta=c_1 P_\alpha S^n_\theta k_0^\theta.
$$ Equality \eqref{q2} follows from linear density of $S^n_\theta k_0^\theta$, $n\geq 0$, in $K^2_\theta$.
 Since $A_{\bar\chi}^{\theta,\alpha}=-A_\psi^{\theta,\alpha}$, by \eqref{e1} and \eqref{q2} we have that $c_1=-c$. Note that $(A_{\bar\chi+c}^{\theta,\alpha})^*=A_{\chi+\bar c}^{\alpha, \theta}$ by Proposition \ref{l3}.
 Using Lemma \ref{l2} we obtain that $A_{\chi+\bar c}^{\alpha,\theta} = 0$ if and only if  $\bar\chi+c\in \overline{ \theta {H}^2}$. Therefore $\varphi = (\psi-c)+(\bar\chi+c)$ is in $\alpha H^2 +\overline{\theta H^2}$.
\end{proof}


\begin{corollary}\label{csymbol}
Let $\alpha\leqslant\theta$ be nonconstant inner functions and let $\ata_\varphi\in\mathcal{T}(\theta,\alpha)$.
For $\varphi\in L^2$ there are functions $\psi\in K^2_\alpha$ and $\chi\in K^2_\theta$ such that $A_\varphi^{\theta,\alpha}=A_{\psi+\bar\chi}^{\theta,\alpha}$. Moreover, $A_{\psi+\bar\chi}^{\theta,\alpha}=A_{\psi_1+\bar\chi_1}^{\theta,\alpha}$ iff $\psi_1=\psi+c k_0^\alpha$, $\chi_1=\chi -\bar ck_0^\theta$ for some constant $c$.
\end{corollary}

\begin{proof}
Let $\varphi = \varphi_+ +\varphi_-$ with $\varphi_+\in H^2$ and $\varphi_-\in  \overline{H^2}$. Put $\psi=P_\alpha \varphi_+$ and $\chi=P_\theta \bar\varphi_-$. Then
$\varphi-\psi-\bar\chi\in \alpha H^2+\overline{\theta  H^2}$. Hence $A_\varphi^{\theta,\alpha}=A_{\psi+\bar\chi}^{\theta,\alpha}$ by Theorem \ref{zero}.

The proof of the second statement is similar to the proof of the corollary on page 499 in \cite{Sarason}. Note firstly that by Proposition \ref{wl} $A_\varphi^{\theta,\alpha}=A_{\psi+\bar\chi}^{\theta,\alpha}=A_{\psi_1+\bar\chi_1}^{\theta,\alpha}$.
On the other hand, if $A_{\psi_1+\bar\chi_1}^{\theta,\alpha}=A_\varphi^{\theta,\alpha}$, then $\psi-\psi_1+\bar\chi-\bar\chi_1\in \alpha H^2+ \overline{\theta H^2}$. Hence $\psi-\psi_1+\bar\chi-\bar\chi_1=\alpha h_1 +\bar\theta \bar h_2$ for some $h_1, h_2\in H^2$. Applying $P_\alpha$ and taking into account that $P_\alpha 1=k_0^\alpha$, we get $\psi-\psi_1=-P_\alpha(\bar\chi-\bar\chi_1) +P_\alpha\bar\theta \bar h_2= -c_1P_\alpha 1=-c_1 k_0^\alpha$ for some constant $c_1$, since $P_\alpha \overline{H^2}$ contains only constant functions. Similarly, we also have that $\bar\psi-\bar\psi_1+\chi-\chi_1=\bar\alpha\bar h_1+\theta h_2$ and applying $P_\theta$ we obtain analogously that
 $\chi-\chi_1=c_2 k_0^\theta$. Since we must have $A_{\psi+\bar\chi}^{\theta,\alpha}-A_{\psi_1+\bar\chi_1}^{\theta,\alpha}=A^{\theta,\alpha}_{-c_1k^\alpha_0+\bar c_2\bar k_0^\theta}=0$, then by Proposition \ref{wl} $-c_1 P_\alpha+\bar c_2 P_\alpha=0$, which implies that $\bar c_2=c_1$.
\end{proof}

The following properties can be immediately obtained from the previous results by taking adjoints.

\begin{corollary}
Let $\aat_\varphi\colon K^2_\alpha\to K^2_\theta$, $\aat_\varphi\in\mathcal{T}(\alpha,\theta)$, $\alpha\leqslant\theta$, $\varphi\in L^2$. Then $\aat_\varphi=0$ iff $\varphi\in\theta H^2+\overline{\alpha H^2}$.
\end{corollary}

\begin{corollary}
Let $\aat_\varphi\colon K^2_\alpha\to K^2_\theta$, $\aat_\varphi\in\mathcal{T}(\alpha,\theta)$, $\alpha\leqslant\theta$, $\varphi\in L^2$. Then there are functions $\psi\in K^2_\alpha$, $\chi\in K^2_\theta$ such that $\aat_\varphi=\aat_{\bar\psi+\chi}$.
\end{corollary}

\section{First characterization in terms of rank--two operators}

In \cite[Theorem 4.1]{Sarason}  a characterization of  truncated Toeplitz operators in $\mathcal{T}(\theta)$ was given by using certain rank two operators defined in terms of the kernel function $k_0^\theta$. Here we obtain an analogous result for asymmetric truncated Toeplitz operators $\mathcal{T}(\theta,\alpha)$ using  the kernel functions $k_0^\alpha$ and $k_0^\theta$.

\begin{theorem}\label{charakt}
Let $\alpha, \theta $ be  nonconstant inner functions such that $\alpha\leqslant\theta$ and let $A:K^2_\theta\to K^2_\alpha$ be a bounded operator. Then $A\in\tcalta$ if and only if there are $\psi\in K_\alpha^2$, $\chi\in K_\theta^2$ such that
\begin{equation}\label{chaar}
A-S_\alpha AS_\theta^*=\psi\otimes k_0^\theta+k_0^\alpha\otimes\chi.
\end{equation}
\end{theorem}

\begin{proof}
Assume that $A\in\tcalta$. Then $A=A_{\psi+\bar\chi}^{\theta,\alpha}$ for some $\psi\in K_\alpha^2$, $\chi\in K_\theta^2$. Note that if $f\in K^\infty_\theta$,  then   $S^*_\theta f\in K^\infty_\theta$ (cf. \eqref{C}.
 Hence
by Proposition \ref{p1}
\begin{multline}\label{eq11}
(A_\psi^{\theta, \alpha}-S_\alpha\ata_\psi S^*_\theta)f=(\ata_\psi-\ata_\psi S_\theta S^*_\theta)f\\=(\ata_\psi(I_{K^2_\theta}-S_\theta S^*_\theta))f=(\ata_\psi k_0^\theta\otimes k_0^\theta)f\\
=(P_\alpha(\psi(1-\overline{\theta(0)}\theta))\otimes k_0^\theta)f=(\psi\otimes k_0^\theta)f.
\end{multline}
Also since $A^\theta_{\bar\chi}$ commutes with $S^*_\theta$ on $K^\infty_\theta$ (\cite[p. 498]{Sarason}) and by Proposition \ref{k-inf}  we have
\begin{align}\label{eq12}
(\ata_{\bar\chi}-&S_\alpha\ata_{\bar\chi}S^*_\theta)f=(\ata_{\bar\chi}-S_\alpha P_\alpha A^\theta_{\bar\chi}S^*_\theta)f\notag
\\&=(\ata_{\bar\chi}-S_\alpha P_\alpha S^*_\theta A^\theta_{\bar\chi})f=(P_\alpha A^\theta_{\bar\chi}-P_\alpha S_\theta S^*_\theta A^\theta_{\bar\chi})f\\&=P_\alpha(I_{K^2_\theta}-S_\theta S_\theta^*)A^\theta_{\bar\chi}f=P_\alpha(k_0^\theta\otimes k_0^\theta)A^\theta_{\bar\chi}f\notag\\&=((P_\alpha k_0^\theta)\otimes(A^\theta_{\chi}k_0^\theta))f=(k_0^\alpha\otimes\chi)f.\notag
\end{align}
Since $K^\infty_\theta$ is dense in $K^2_\theta$, we obtain \eqref{chaar}.

For the converse implication note that
for $\psi\in K^2_\alpha$, $\chi\in K^2_\theta$ and for all $f\in K^\infty_\theta$, $g\in K^\infty_\alpha$, we have
\begin{equation}\label{eq21}
\langle \ata_{\psi+\bar\chi}f,g\rangle =\sum_{n=0}^\infty(\langle f,S_\theta^n k_0^\theta\rangle\langle S_\alpha^n\psi,g\rangle+\langle f,S^n_\theta \chi\rangle\langle S_\alpha^n k_0^\alpha,g\rangle).
\end{equation}
Indeed, we obtain \eqref{eq21} as in the proof of \cite[Lemma 4.2]{Sarason} taking $\psi\in K^2_\alpha$ and $\chi\in K^2_\theta$ instead of $\psi,\chi\in K^2_u$  and  $S_\alpha$, $S^*_\theta$ instead of $S_u$, $S^*_u$, respectively.
The convergence $\langle A^{\theta,\alpha}_{\psi+\bar\chi}(S^*_\theta)^{N+1} f, {S^*_\alpha}^{N+1}g\rangle \to 0$ as $N\to \infty$  follows from
\begin{align*}\langle S_\alpha^N A^{\theta,\alpha}_{\psi+\bar\chi}(S^*_\theta)^N f,g\rangle&=\langle S_\alpha^N A^{\theta,\alpha}_{\psi}(S^*_\theta)^N f,g\rangle+\langle S_\alpha^N A^{\theta,\alpha}_{\bar\chi}(S^*_\theta)^N f,g\rangle \\
&=\langle (S^*_\theta)^N f, A^{\alpha,\theta}_{\bar\psi}(S_\alpha^*)^N g\rangle+\langle S_\alpha^N P_\alpha A^{\theta}_{\bar\chi}(S^*_\theta)^N f,g\rangle
\\&=\langle (S^*_\theta)^N f,(S^*_\theta)^N A^{\alpha,\theta}_{\bar\psi} g \rangle+
\langle (S^*_\theta)^N A^{\theta}_{\bar\chi} f,P_\alpha(S^*_\alpha)^N g \rangle \\&=\langle (S^*_\theta)^N f,(S^*_\theta)^N A^{\alpha,\theta}_{\bar\psi} g \rangle+
\langle (S^*_\theta)^N A^{\theta}_{\bar\chi} f,(S^*_\alpha)^N g \rangle,
\end{align*}
where the last expression tends to zero, when $N\to \infty$, by the strong convergence $(S^*)^N\to 0$.

Now assume that a bounded operator $A\colon K^2_\theta\to K^2_\alpha$ satisfies \eqref{chaar}. Following the proof of Theorem 4.1 \cite{Sarason}  we obtain
 \begin{equation}\label{eq22}A=\sum_{n=0}^\infty(S_\alpha^n\psi\otimes S_\theta^n k_0^\theta+S_\alpha^n k_0^\alpha\otimes S^n_\theta \chi),
 \end{equation} taking into account that  ${S^*_\theta}^n\to 0$ in the strong operator topology.
 Hence, comparing \eqref{eq21} and \eqref{eq22}, we conclude that the right hand side of \eqref{eq22} is equal to $\ata_{\psi+\bar\chi}$.
\end{proof}

We can obtain a similar characterization for operators from $\mathcal{T}(\alpha,\theta)$ by taking adjoints in \eqref{chaar}. Namely, we have:

\begin{corollary}
Let $\alpha, \theta$ be nonconstant inner functions such that $\alpha\leqslant\theta$ and let $A\colon K^2_\alpha\to K^2_\theta$ be a bounded operator. Then
$A\in\mathcal{T}(\alpha,\theta)$ if and only if there are $\psi\in K^2_\alpha$, $\chi\in K^2_\theta$ such that
\begin{equation}
A-S_\theta AS^*_\alpha=k_0^\theta\otimes\psi +\chi\otimes k_0^\alpha.
\end{equation}
\end{corollary}

\section{Second characterization in terms of rank--two operators}

Sarason obtained also a characterization for truncated Toeplitz operators belonging to $\mathcal{T}(\theta)$ involving the function $\tilde k_0^\theta=C_\theta k_0^\theta$ instead of $ k_0^\theta$, by a simple application of the conjugation $C_\theta$ to the result of Theorem \ref{charakt} in the case $\alpha=\theta$. Here we will show that an analogous result holds for operators belonging to $\mathcal{T}(\theta,\alpha)$, $\alpha\leqslant \theta$. However, we cannot use the same reasoning for $\alpha\neq \theta$, and, as we will see, in the case of asymmetric truncated Toeplitz operators the situation is  more complex and, though a similar characterization is obtained, the proof is significantly more involved.

\begin{theorem}\label{char-sp}
Let $\alpha, \theta$ be nonconstant inner functions such that $\alpha \leqslant \theta$, and let $A\colon K^2_\theta\to K^2_\alpha$ be a bounded operator. Then $A\in\mathcal{T}(\theta,\alpha)$ if and only if there are $\mu\in K^2_\alpha$ and $\nu\in K^2_\theta$ such that
\begin{equation}\label{cstar}A-S_\alpha^* A S_\theta=\mu\otimes \tilde{k}_0^\theta+\tilde{k}_0^\alpha\otimes \nu.\end{equation}
\end{theorem}

\begin{proof} Assume firstly that $A\in \mathcal{T}(\theta,\alpha)$. By Theorem \ref{charakt}, there are $\psi\in K^2_\alpha$ and $\chi\in K^2_\theta$ such that
$$A-S_{\alpha}AS^*_{\theta}=\psi\otimes k_0^\theta +k_0^\alpha\otimes \chi.$$
Moreover, $A=A_{\psi+\bar\chi}^{\theta,\alpha}$. Hence
$$S^*_{\alpha}(A-S_{\alpha}AS^*_{\theta})S_{\theta}= S^*_{\alpha}\psi \otimes S^*_{\theta}k_{0}^{\theta}+S^*_{\alpha}k_{0}^{\alpha} \otimes S^*_{\theta}\chi.$$
 Using $I_{K^2_\alpha}-S^*_\alpha S_\alpha=\tilde{k}_0^\alpha\otimes\tilde{k}_0^\alpha$ and $I_{K^2_\theta}-S^*_\theta S_\theta=\tilde{k}_0^\theta\otimes \tilde{k}_0^\theta$ after some calculations we get
\begin{multline}\label{star}
A-S^*_{\alpha}AS_{\theta}=\tilde{k}_0^{\alpha} \otimes A^*\tilde{k}_0^{\alpha} + A \tilde{k}_0^{\theta} \otimes \tilde{k}_0^{\theta} - \langle A \tilde{k}_0^{\theta}, \tilde{k}_0^{\alpha}\rangle \tilde{k}_0^{\alpha} \otimes \tilde{k}_0^{\theta}\\ - S^*_{\alpha}\psi \otimes S^*_{\theta}k_0^{\theta}- S^*_{\alpha}k_0^{\alpha} \otimes S^*_{\theta}\chi.
\end{multline}
Since $A=A_{\psi+\bar\chi}^{\theta,\alpha}$ with $\psi\in K^2_\alpha,\,\chi\in K^2_\theta$,
 by Proposition \ref{kzero} (7)-(8), and taking into account that $\chi=\chi_{\frac{\theta}{\alpha}}+\frac{\theta}{\alpha}\chi_\alpha$, where $\chi_{\frac{\theta}{\alpha}}\in K^2_{\frac{\theta}{\alpha}}$ and $\chi_\alpha\in K^2_\alpha$,
\begin{multline}
\tilde{k}_0^{\alpha} \otimes A^{\alpha,\theta}_{\bar\psi+\chi}\tilde{k}_0^{\alpha}= \tilde{k}_0^{\alpha} \otimes C_{\alpha} \psi + \tilde{k}_0^{\alpha} \otimes P_{\theta} C_{\theta} (\bar{\chi}\tfrac{\theta}{\alpha}) - \overline{\alpha(0)} \tilde{k}_0^{\alpha} \otimes S^*_{\theta}\chi\\=\tilde{k}_0^{\alpha} \otimes C_{\alpha} \psi + \tilde{k}_0^{\alpha} \otimes S^* (\alpha\chi_{\frac{\theta}{\alpha}})+\langle k_0^\alpha,\chi_\alpha\rangle \tilde k_0^\alpha\otimes \tilde k_0^\theta - \overline{\alpha(0)} \tilde{k}_0^{\alpha} \otimes S^*_{\theta}\chi.\end{multline}
 On the other hand, by Proposition  \ref{kzero} (3) and (4),
 $$A^{\theta,\alpha}_{\psi+\bar\chi} \tilde{k}_0^{\theta} \otimes \tilde{k}_0^{\theta}=\tfrac{\theta}{\alpha}(0)\psi(0)\tilde{k}_0^{\alpha} \otimes \tilde{k}_0^{\theta} - \theta(0) S^*_{\alpha}\psi\otimes \tilde{k}_0^{\theta} + C_{\alpha}P_{\alpha} (\chi\tfrac{\bar{\theta}}{\bar{\alpha}}) \otimes \tilde{k}_0^{\theta}.$$
 Since $\psi \in K^2_\alpha$, we have $\alpha \bar\psi\in H_0^2$ and thus $(\alpha\bar{\psi})(0)=0$. Hence, since $C_\alpha S^*_\alpha=S_\alpha C_\alpha$  (\cite[Lemma 2.1]{Sarason}), we get
      \begin{multline*}
      \langle S^*_{\alpha}\psi, \tilde{k}_0^{\alpha} \rangle = \langle k_0^{\alpha}, C_{\alpha} S^*_{\alpha}\psi \rangle=\langle  k_0^{\alpha}, S_{\alpha} C_{\alpha}\psi \rangle=\langle  k_0^{\alpha}, S_{\alpha} \alpha\bar{z}\bar{\psi} \rangle\\=\langle  k_0^{\alpha}, \alpha \bar{\psi}\rangle=\langle 1-\overline{\alpha(0)} {\alpha}, \alpha \bar{\psi}\rangle=\langle 1,\alpha\bar\psi\rangle-\overline{\alpha(0)}\langle 1, \bar{\psi}\rangle=
      -\overline{\alpha(0)}\psi(0).
      \end{multline*}
      Therefore by Proposition \ref{kzero} (3) and (4) we have
      \begin{align*}
      \langle A_{\psi+\bar{\chi}}^{\theta,\alpha} \tilde{k}_0^{\theta}, \tilde{k}_0^{\alpha} \rangle &= \langle \tfrac{\theta}{\alpha}(0) \psi(0)\tilde{k}_0^{\alpha}-\theta(0) S_\alpha^*\psi +C_{\alpha}P_{\alpha} (\tfrac{\bar{\theta}}{\bar{\alpha}}\chi),\tilde{k}_0^{\alpha}
       \rangle  \\ &= \tfrac{\theta}{\alpha}(0) \psi(0) \|\tilde{k}_0^{\alpha}\|^2 +\theta(0)\overline{\alpha(0)}{\psi(0)}+  \langle C_{\alpha} P_{\alpha}(\tfrac{\bar{\theta}}{\bar{\alpha}}\chi), \tilde{k}_0^{\alpha} \rangle \\ &= \tfrac{\theta}{\alpha}(0) \psi(0) (1-|\alpha(0)|^2) +\theta(0)\overline{\alpha(0)}{\psi(0)}+\langle C_\alpha\chi_\alpha,\tilde k_0^{\alpha}\rangle \\ &=\tfrac{\theta}{\alpha}(0) \psi(0)+\langle k_0^{\alpha},\chi_\alpha\rangle.
      \end{align*}

   Moreover, by \eqref{D} we get  $S^*_{\alpha}k_0^{\alpha} \otimes S^*_{\theta}\chi=-\overline{\alpha(0)}\tilde{k}_0^{\alpha} \otimes S^*_{\theta}\chi$
 and $S^*_{\alpha}\psi \otimes S^*_{\theta}k_0^{\theta}=-\theta(0)  S^*_{\alpha}\psi \otimes \tilde{k}_0^{\theta}$.
Hence the right hand side of equation \eqref{star}, after  simplification,  is equal to
$$
\tilde{k}_0^{\alpha} \otimes C_{\alpha} \psi + \tilde{k}_0^{\alpha} \otimes S^* (\alpha\chi_{\frac{\theta}{\alpha}})
+ C_{\alpha}P_{\alpha} (\chi\tfrac{\bar{\theta}}{\bar{\alpha}}) \otimes \tilde{k}_0^{\theta},$$
which can be written as $$\tilde{k}_0^{\alpha} \otimes (C_{\alpha} \psi+S^* (\alpha\, \chi_{\frac{\theta}{\alpha}}))+C_{\alpha}\,\chi_{\alpha} \otimes \tilde{k}_0^{\theta}. $$
Note that $C_{\alpha}\,\chi_{\alpha}\in K^2_\alpha$, $C_{\alpha} \psi\in K^2_\alpha\subset K^2_\theta$. Moreover, $\alpha\chi_{\frac{\theta}{\alpha}}\in K^2_\theta$, which implies that $S^* (\alpha\chi_{\frac{\theta}{\alpha}})=S_\theta^* (\alpha\chi_{\frac{\theta}{\alpha}})\in  K^2_\theta$.
Putting
\begin{equation}\label{symbol}
\mu=
C_\alpha \,\chi_\alpha\in K^2_\alpha\text{ and }\nu=
C_{\alpha} \psi+S^*_\theta (\alpha \chi_{\frac{\theta}{\alpha}})\in K^2_\theta
\end{equation}
 we obtain
\begin{equation}\label{opa}A-S_\alpha^* A S_\theta=\mu\otimes \tilde{k}_0^\theta+\tilde{k}_0^\alpha\otimes \nu.\end{equation}

For the reverse implication assume that a bounded operator $A$ satisfies the equality \eqref{cstar} for $\mu\in K^2_\alpha$, $\nu\in K^2_\theta$.
Then.
\begin{equation}
S_\alpha AS^*_\theta-S_\alpha S^*_\alpha AS_\theta S^*_\theta=S_\alpha\mu\otimes S_\theta \tilde k_0^\theta+S_\alpha\tilde k_0^\alpha\otimes S_\theta\nu
\end{equation}
Using $I_{K^2_\alpha}-S_{\alpha}S^*_{\alpha}=k^{\alpha}_0 \otimes k^{\alpha}_0$, $I_{K^2_\theta}-S_{\theta}S^*_{\theta}=k^{\theta}_0 \otimes k^{\theta}_0$ and  \eqref{D} we obtain
\begin{multline*}
A-S_\alpha AS^*_\theta=
Ak_0^\theta\otimes k_0^\theta+k_0^\alpha\otimes A^*k_0^\alpha-\langle Ak_0^\theta,k_0^\alpha\rangle k_0^\alpha\otimes k_0^\theta\\+\overline{\theta(0)}S_\alpha\mu\otimes k_0^\theta+\alpha(0)k_0^\alpha\otimes S_\theta\nu.
\end{multline*}
Hence
\begin{equation*}
A-S_\alpha AS^*_\theta=(Ak_0^\theta-\langle Ak_0^\theta,k_0^\alpha\rangle k_0^\alpha+\overline{\theta(0)}S_\alpha\mu)\otimes k_0^\theta+k_0^\alpha\otimes(A^*k_0^\alpha+\overline{\alpha(0)}S_\theta\nu),
\end{equation*}
which by Theorem \ref{charakt} implies that $A\in\mathcal{T}(\theta,\alpha)$.
\end{proof}

By taking adjoints in \eqref{cstar} we obtain a similar
characterization for operators from $\mathcal{T}(\alpha,\theta)$:

\begin{corollary}
Let $\alpha, \theta$ be nonconstant inner functions such that $\alpha\leqslant\theta$, and let $A\colon K^2_\alpha\to K^2_\theta$ be a bounded operator. Then
$A\in\mathcal{T}(\alpha,\theta)$ if and only if there are $\mu\in K^2_\alpha$, $\nu\in K^2_\theta$ such that
\begin{equation}
A-S_\theta^* AS_\alpha=\tilde k_0^\theta\otimes\mu +\nu\otimes \tilde k_0^\alpha.
\end{equation}
\end{corollary}

It is clear that if an asymmetric truncated Toeplitz operator $A$ satisfies equation \eqref{cstar} with some
$\mu$, $\nu$, then that equation is also satisfied if  $\mu$, $\nu$ are replaced by
\begin{equation}\label{6.2a}
\mu^\prime=\mu+\bar b\,\tilde k_0^\alpha\,, \,\nu^\prime=\nu- b\,\tilde k_0^\theta,
\end{equation}
 respectively, for any $b\in \mathbb{C}$. On the other hand, it is also true that the symbol of $A=A_{\psi+\bar\chi}^{\theta,\alpha}\in\mathcal{T}(\theta,\alpha)$ is not unique, and by
Corollary \ref{csymbol} we can replace $\psi\in K^2_\alpha$ and $\chi\in K^2_\theta$ by
\begin{equation}\label{6.2b}
\psi^\prime=\psi+c\, k_0^\alpha\in K^2_\alpha\,,\, \chi^\prime=\chi-\bar c\, k_0^\theta\in K^2_\theta,
\end{equation}
  respectively, for any $c\in \mathbb{C}$.
Using \eqref{symbol}, it is easy to see that the following relation between the freedom of choice of $\mu$, $\nu$ on the one hand, and $\psi$, $\chi$ on the other, holds.

\begin{corollary}\label{wn3}
Let $\mu\in K^2_\alpha$ and $\nu\in K^2_\theta$ be defined by \eqref{symbol} for given $\psi\in K^2_\alpha$ and $\chi\in K^2_\theta$, and let $\mu^\prime\in K^2_\alpha$ and $\nu^\prime\in K^2_\theta$ be defined analogously for $\psi^\prime\in K^2_\alpha$ and $\chi^\prime\in K^2_\theta$. If \eqref{6.2b} holds, then
\[\mu^\prime=\mu-c\ \tfrac\theta\alpha(0)\tilde k_0^\alpha\,\,,\,\,\nu^\prime=\nu+
\bar c\ \overline{\tfrac\theta\alpha(0)}\tilde k_0^\theta.\]
 \end{corollary}




The examples below illustrate the result of Theorem \ref{char-sp} in the case of Toeplitz matrices.

\begin{example}
Let us consider $\alpha=z^2$, $\theta=z^5$ and a Toeplitz operator $A=A_{\psi+\bar\chi}^{z^5}$. Assume that $\psi=a_0+a_1 z$ and $\chi=\bar b_0+\bar b_{-1} z+\bar b_{-2}z^2+\bar b_{-3}z^3+\bar b_{-4}z^4=(\bar b_0+\bar b_{-1} z+\bar b_{-2}z^2)+z^3(\bar b_{-3}+\bar b_{-4}z)$. Then $C_{z^2}\psi=\bar a_1+\bar a_0 z$, $C_{z^2}P_{z^2}\bar z^3\chi=b_{-4}+b_{-3}z$ and $S^*(z^2(\bar b_0+\bar b_{-1} z+\bar b_{-2}z^2))=\bar b_0 z+\bar b_{-1} z^2+\bar b_{-2}z^3$.
Note  that $A-S^*_{z^2}AS_{z^5}$ has a matrix representation
$$\left(
    \begin{array}{ccccc}
      0& 0 & 0 & 0 & b_{-4} \\
      a_1 & a_0+b_0 & b_{-1} & b_{-2} & b_{-3} \\
    \end{array}
  \right),$$
  which can be expressed as $$(b_{-4}+b_{-3}z)\otimes z^4+z\otimes (\bar a_1+(\bar a_0 +\bar b_0) z+\bar b_{-1} z^2+\bar b_{-2}z^3).$$

On the other hand, let  $A-S^*_{z^2}AS_{z^5}$ have a matrix representation
$$\left(
    \begin{array}{ccccc}
      0& 0 & 0 & 0 & b_{0} \\
      a_0 & a_1 & a_{2} & a_{3} & a_4+b_{1} \\
    \end{array}
  \right),$$
  which can be expressed as $$\mu\otimes z^4+z\otimes \nu=(b_{0}+b_{1}z)\otimes z^4+z\otimes (\bar a_0+\bar a_1  z+\bar a_{2} z^2+\bar a_{3}z^3+\bar a_4 z^4).$$
Note that  $\nu=\nu_{z^2}+z^2\nu_{z^3}=(\bar a_0+\bar a_1  z)+z^2(\bar a_{2} +\bar a_{3}z+\bar a_4 z^2)$.
Then
$\psi=C_{z^2}P_{z^2}\nu=a_1+ a_0 z$ and $\chi=\bar a_2z+\bar a_3z^2+(\bar b_1+\bar a_4)z^3+\bar b_0z^4$.
Hence by Theorem \ref{charakt} we have
\begin{equation}\label{eq13}
A-S_{z^2}AS^*_{z^5}=(a_1+ a_0 z)\otimes 1+1\otimes (\bar a_2z+\bar a_3z^2+(\bar b_1+\bar a_4)z^3+\bar b_0z^4)
\end{equation}
 Requiring that $\nu_{z^3}$ is orthogonal to $z^2$ (see the proof of  Theorem \ref{pr1}) determines that $a_4=0$.

On the other hand, we have some freedom in defining $\psi$ and $\chi$; namely
$\psi_1=s+ a_0 z$ and $\chi_1=\bar t+\bar a_2z+\bar a_3z^2+(\bar b_1+\bar a_4)z^3+\bar b_0z^4$ also satisfy \eqref{eq13} if we assume that $t+s=a_1$.
\end{example}

\begin{example}
Let us now take $\alpha=z^3$, $\theta= z^3(({\lambda-z})/(1-\bar{\lambda}z))^{2}$,
$\lambda \in\mathbb{D}$ and consider the operator $A=A^{\theta,\alpha}_{\psi+\bar{\chi}}$, where $\psi=a_0+a_1z+a_2z^2\in K^2_\alpha$ and $\chi=(\bar b_0+\bar b_1z+\bar b_2z^2+\bar b_3 z^3+\bar b_4z^4)({1-\bar{\lambda}z})^{-2}\in K^2_\theta$ (see \cite[Corollary 5.7.3]{GMR}). Then by Theorem \ref{char-sp} and \eqref{symbol}
 $$A-S^*_{\alpha}AS_{\theta}=\mu\otimes(\lambda^2z^2-2\lambda z^3+z^4)({1-\bar{\lambda}z})^{-2}+z^2\otimes\nu,$$
 where $\mu=b_4+(b_3+2\bar{\lambda}b_4)z+(b_2+3\bar{\lambda}^2b_4+\bar{\lambda}b_3)z^2$
and $\nu=\big(\bar a_2+(\bar a_1-2\bar{\lambda}\bar a_2)z+(\bar{b}_0+\bar a_0-2\bar{\lambda}\bar{a}_1
+\bar{\lambda}\bar{a}_2)z^2+(\bar{b}_1+
\bar{\lambda}^2\bar{a}_1-2\bar{\lambda}\bar{a}_0)z^3+\bar{\lambda}^2\bar{a}_0z^4\big)(1-\bar{\lambda}z)^{-2}.$
\end{example}

\section{Characterizations in terms of rank--one operators}

Our aim now is to describe the classes of symbols of an operator $A\in\mathcal{T}(\theta,\alpha)$ for which the right hand side of
\eqref{cstar} is a rank one operator. The corresponding question regarding the equation \eqref{chaar} is trivial by Corollary \ref{csymbol}, since the right side of \eqref{chaar} is a rank one operator if and only if $\psi=c\cdot k_0^\alpha$ or $\chi=c\cdot k_0^\theta$ with $c\in\mathbb C$. In the case $\alpha=\theta$ the question regarding the equality \eqref{cstar} also has an easy answer, since the relation between the symbols in  \eqref{chaar} and \eqref{cstar} is $\psi=C_\theta \nu$ and $\chi=C_\theta \mu$. For $\alpha\neq\theta$, we need the following lemma.

\begin{lemma}\label{l4}
Let $\alpha\leqslant \theta$ be nonconstant inner functions and $\ata_{\psi+\bar\chi}\in \mathcal{T}(\theta,\alpha)$. Assume that
$\psi\in K^2_\alpha$ and $\chi=\chi_{\frac{\theta}{\alpha}}+\frac{\theta}{\alpha}\chi_\alpha\in K^2_\theta=K^2_{\frac{\theta}{\alpha}}\oplus\frac{\theta}{\alpha}K^2_\alpha$. Then
\begin{enumerate}
\item $P_{\mathbb{C}\tilde k_0^\theta}S^*_\theta(\alpha\chi_{\frac{\theta}{\alpha}})=-\overline{\tfrac{\theta}{\alpha}(0)}\chi_{\frac{\theta}{\alpha}}(0)\|\tilde k_0^\theta\|^{-2}\tilde k_0^\theta$,
\item $P_{\mathbb{C}\tilde k_0^\theta}C_\alpha\psi=\overline{\frac{\theta}{\alpha}(0)\psi(0)}\|\tilde k_0^\theta\|^{-2}\tilde k_0^\theta$,
\item $P_\alpha S^*_\theta(\alpha\chi_{\frac{\theta}{\alpha}})=P_{\mathbb{C}\tilde k_0^\alpha}S^*_\theta(\alpha\chi_{\frac{\theta}{\alpha}})=\chi_{\frac{\theta}{\alpha}}(0)\tilde k_0^\alpha$,
\item $(P_\theta-P_\alpha)S^*_\theta(\alpha\chi_{\frac{\theta}{\alpha}})=\alpha S^*\chi_{\frac{\theta}{\alpha}},$
\item $ P_{\mathbb{C}\tilde k_0^\alpha}C_\alpha P_\alpha(\chi \tfrac{\bar\theta}{\bar \alpha})= \overline {\chi_\alpha(0)}\|\tilde k_0^\alpha\|^{-2}\tilde k_0^\alpha,$
\item $P_{\mathbb{C}\tilde k_0^\alpha}C_\alpha\psi=\overline{\psi(0)}\|\tilde k_0^\alpha\|^{-2}\tilde k_0^\alpha$.
\end{enumerate}
\end{lemma}

\begin{proof} To prove (1) it is enough to calculate
\begin{multline*}\langle S^*(\alpha\chi_{\frac{\theta}{\alpha}}),\tilde k_0^\theta\rangle=\langle \bar z(\alpha\chi_{\frac{\theta}{\alpha}}-\alpha(0)\chi_{\frac{\theta}{\alpha}}(0)),\bar z (\theta-\theta(0))\rangle\\=\langle \chi_{\frac{\theta}{\alpha}},\tfrac{\theta}{\alpha}\rangle-\alpha(0)\chi_{\frac{\theta}{\alpha}}(0)\langle 1,\theta\rangle-\overline{\theta(0)}\langle \alpha \chi_{\frac{\theta}{\alpha}},1\rangle+\alpha(0)\overline{\theta(0)}\chi_{\frac{\theta}{\alpha}}(0)=-\overline{\tfrac{\theta}{\alpha}(0)}\chi_{\frac{\theta}{\alpha}}(0).
\end{multline*}
To show (2) note that by Proposition \ref{cos}
 \begin{multline*}
 \langle C_\alpha \psi, \tilde k_0^\theta\rangle= \langle C_\alpha \psi, P_\alpha\tilde k_0^\theta\rangle=
  \langle C_\alpha \psi, \tfrac{\theta}{\alpha}(0)\tilde k_0^\alpha\rangle\\= \langle C_\alpha \tfrac{\theta}{\alpha}(0)\tilde k_0^\alpha,\psi\rangle=\overline{\tfrac{\theta}{\alpha}(0)}\langle k_0^\alpha,\psi\rangle=\overline{\tfrac{\theta}{\alpha}(0)}\ \overline{\psi(0)}.
  \end{multline*}
 The equalities in (3) follow from
 \begin{multline*}
 P_\alpha S^*_\theta(\alpha\chi_{\frac{\theta}{\alpha}})=  P_\alpha (\bar z\alpha\chi_{\frac{\theta}{\alpha}})=\alpha P^-\bar\alpha\bar z\alpha\chi_{\frac{\theta}{\alpha}}-P^-\bar z\alpha\chi_{\frac{\theta}{\alpha}}
 =\alpha P^-\bar z\chi_{\frac{\theta}{\alpha}}-\alpha(0)\chi_{\frac{\theta}{\alpha}}(0)\bar z\\=\alpha \chi_{\frac{\theta}{\alpha}}(0)\bar z-\alpha(0)\chi_{\frac{\theta}{\alpha}}(0)\bar z=\chi_{\frac{\theta}{\alpha}}(0)\tilde k_0^\alpha=P_{\mathbb{C}\tilde k_0^\alpha}S^*_\theta(\alpha\chi_{\frac{\theta}{\alpha}}).
 \end{multline*}
 To see (4) it is enough to calculate
  \begin{multline*}
(P_\theta- P_\alpha) S^*_\theta(\alpha\chi_{\frac{\theta}{\alpha}})=  \alpha P_{\frac{\theta}{\alpha}}\bar\alpha
S^*_\theta(\alpha\chi_{\frac{\theta}{\alpha}})= \alpha P_{\frac{\theta}{\alpha}}(\bar\alpha\bar z
(\alpha\chi_{\frac{\theta}{\alpha}}-\alpha(0)\chi_{\frac{\theta}{\alpha}}(0)))\\= \alpha P_{\frac{\theta}{\alpha}}(\bar z\chi_{\frac{\theta}{\alpha}})=
\alpha S^*_{\frac{\theta}{\alpha}}\chi_{\frac{\theta}{\alpha}}=\alpha S^*\chi_{\frac{\theta}{\alpha}}.
 \end{multline*}
 The proof of (5) follows from
 $\langle C_\alpha P_\alpha(\chi \tfrac{\bar\theta}{\bar \alpha}), \tilde k_0^\alpha\rangle=\langle k_0^\alpha, \chi_\alpha\rangle = \overline {\chi_\alpha(0)}$, and to show that (6) holds  note that
 $\langle C_\alpha\psi, \tilde k_0^\alpha\rangle=\langle k_0^\alpha, \psi\rangle=\overline{\psi(0)}.$
\end{proof}

\begin{theorem}\label{pr2}
Let $\alpha\leqslant \theta$ be nonconstant inner functions and let $\ata_{\psi+\bar\chi}\in \mathcal{T}(\theta,\alpha)$, where $\psi\in K^2_\alpha$ and $\chi\in K^2_\theta$.
Then
\begin{enumerate}
\item $\ata_{\psi+\bar\chi}-S^*_\alpha \ata_{\psi+\bar\chi}S_\theta=\mu\otimes \tilde k^\theta_0$ for $\mu\in K^2_\alpha$ if and only if there is $s\in\mathbb{C}$ such that $\psi=sk_0^\alpha$, $P_{\frac{\theta}{\alpha}}\chi=-\bar sk_0^{\frac\theta\alpha}$,
\item    $\ata_{\psi+\bar\chi}-S^*_\alpha \ata_{\psi+\bar\chi}S_\theta=\tilde k^\alpha_0\otimes \nu$ for $\nu\in K^2_\theta$ if and only if $P_\alpha(\chi\frac{\bar\theta}{\bar\alpha})={const}\cdot k_0^\alpha$.
\end{enumerate}
\end{theorem}

\begin{proof}
Assume first that  $\ata_{\psi+\bar\chi}-S^*_\alpha \ata_{\psi+\bar\chi}S_\theta=\mu\otimes \tilde k^\theta_0$ for $\mu\in K^2_\alpha$. Now,
  the right hand side of the equation \eqref{cstar} reduces to $\mu\otimes \tilde k^\theta_0$ if and only if $\nu=c\,\cdot \tilde k_0^\theta$ with $c\in\mathbb C$, which is equivalent to $\nu-P_{\mathbb{C}\tilde k_0^\theta}\nu=0.$ Let
  $\chi=\chi_{\frac{\theta}{\alpha}}+\frac{\theta}{\alpha}\chi_\alpha\in K^2_\theta=K^2_{\frac{\theta}{\alpha}}\oplus\frac{\theta}{\alpha}K^2_\alpha$.
 By formulas  \eqref{symbol} and Lemma \ref{l4} we thus have that
\begin{multline}\label{e7}0=C_\alpha\psi+S_\theta^*(\alpha\chi_{\frac{\theta}{\alpha}})+\overline{\tfrac{\theta}{\alpha}(0)}(\chi_\frac{\theta}{\alpha}(0)-\overline{\psi(0)})\|k_0^\theta\|^{-2}\tilde k_0^\theta\\=\big(C_\alpha\psi+(\chi_{\frac{\theta}{\alpha}}(0)+|\tfrac{\theta}{\alpha}(0)|^2(\chi_{\frac{\theta}{\alpha}}(0)-\overline{\psi(0)})\|k_0^\theta\|^{-2})\tilde k_0^\alpha\big)\\+\alpha (S_\theta^*\chi_{\frac{\theta}{\alpha}}+\overline{\tfrac{\theta}{\alpha}(0)}(\chi_{\frac{\theta}{\alpha}}(0)-\overline{\psi(0)})\|k_0^\theta\|^{-2}\tilde k_0^{\frac{\theta}{\alpha}}).\end{multline}
Since the right hand side of \eqref{e7} is an orthogonal sum,  each term must be zero. Thus, for $s=-\overline{\chi_{\frac{\theta}{\alpha}}(0)}-|\tfrac{\theta}{\alpha}(0)|^2(\overline{\chi_{\frac{\theta}{\alpha}}(0)}-\psi(0))\|k_0^\theta\|^{-2}$, we get $\psi=s k_0^\alpha$. Moreover, $S_\theta^*\chi_{\frac{\theta}{\alpha}}=-\overline{\tfrac{\theta}{\alpha}(0)}\big(\chi_{\frac{\theta}{\alpha}}(0)-\overline{\psi(0)}\big)\|k_0^\theta\|^{-2}\tilde k_0^{\frac{\theta}{\alpha}}$. Hence, proceeding as in \eqref{chita}, we get
$$
\chi_{\frac{\theta}{\alpha}}=\chi_{\frac{\theta}{\alpha}}(0)k_0^{\frac{\theta}{\alpha}}-\overline{\tfrac{\theta}{\alpha}(0)}(\chi_{\frac{\theta}{\alpha}}(0)-\overline{\psi(0)})\|k_0^\theta\|^{-2}S_{\frac{\theta}{\alpha}}\tilde k_0^{\frac{\theta}{\alpha}}=-\bar s k_0^{\frac{\theta}{\alpha}},
$$
by Lemma \ref{D}.

The converse is immediate from \eqref{symbol}.
The proof of (2) is analogous.
\end{proof}

\begin{remark}\label{rem1} When the right hand side of the characterization \eqref{chaar} reduces to a rank one operator
  $const\cdot k_0^\alpha\otimes k_0^\theta$ it is immediate that this operator can be expressed in terms of the symbol $\psi+\bar \chi$ as
    \[const\cdot k_0^\alpha\otimes k_0^\theta=P_{\mathbb{C} k_0^\alpha}\psi\otimes  k_0^\theta+  k_0^\alpha\otimes P_{\mathbb{C} k_0^\theta}\chi=\big(\psi(0)\|\tilde k_0^\alpha\|^{-2}+\overline{\chi(0)}\|\tilde k_0^\theta\|^{-2}\big) k_0^\alpha\otimes k_0^\theta.\]
     It might be of independent interest to consider the case when the right hand side in the equation \eqref{cstar}
reduces to a rank one operator  $const\cdot \tilde k_0^\alpha\otimes\tilde k_0^\theta$.
 By equations \eqref{symbol} and Lemma \ref{l4} (1-2) and (5) this operator can be expressed in terms of the symbol $\psi+\bar \chi$ as
   \begin{multline*}const\cdot\tilde k_0^\alpha\otimes\tilde k_0^\theta = P_{\mathbb{C}\tilde k_0^\alpha}\mu\otimes \tilde k_0^\theta+ \tilde k_0^\alpha\otimes P_{\mathbb{C}\tilde k_0^\theta}\nu=\\
    \left(\overline {\chi_\alpha(0)}\|\tilde k_0^\alpha\|^{-2}+\tfrac{\theta}{\alpha}(0)(\psi(0)-\overline{\chi_{\frac \theta\alpha}(0)} \|\tilde k_0^\theta\|^{-2})\right)\ \tilde k_0^\alpha\otimes\tilde k_0^\theta.\end{multline*}
A similar question can be asked regarding the case when the right hand side of the equation \eqref{cstar}
reduces to a rank one operator $const\cdot\tilde k_0^\alpha\otimes\tilde k_0^\alpha$. By
 equations \eqref{symbol}, Proposition \ref{cos} (4) and Lemma \ref{l4} (3), (5) and (6) we have
\begin{multline*}const\cdot\tilde k_0^\alpha\otimes\tilde k_0^\alpha=P_{\mathbb{C}\tilde k_0^\alpha}\mu\otimes \tilde k_0^\alpha+ \tilde k_0^\alpha\otimes P_{\mathbb{C}\tilde k_0^\alpha}\nu=\\\left(\overline{\chi(0)}-\overline{\chi_{\frac{\theta}{\alpha}}(0)}|\alpha(0)|^2 +\psi(0)\right)\, \|k_0^\alpha\|^{-2}\  \tilde k_0^\alpha\otimes\tilde k_0^\alpha.\end{multline*}
\end{remark}

\section{An inverse problem: from the operator to the symbol}

In the case of a classical Toeplitz operator $T_\varphi$ on $H^2$,  the (unique) symbol  $\varphi$ can be obtained from the operator by the formula $\lim_{n\to\infty}\bar z^nT_\varphi z^n$. In the case of a truncated Toeplitz operator, i.e., of the form $\aat_\varphi$ with $\alpha=\theta$, one can obtain a symbol belonging to $H^2+\overline{H^2}$ from the action of $A_\varphi^\theta$ on $k^\theta_0$ and $\tilde k^\theta_0$
(\cite{BCT}). A similar result can be obtained for an asymmetric truncated Toeplitz operator $A\in\mathcal{T}(\theta,\alpha)$ by considering the action of the operator $A$ and its adjoint on reproducing kernel functions of the same kind.

Note firstly that if $A$ has a symbol $\psi+\bar\chi$ with $\psi\in K^2_\alpha,\, \chi\in K^2_{\theta}$, writing $\chi=\chi_{\frac{\theta}{\alpha}}+\frac{\theta}{\alpha}\chi_\alpha$, where $\chi_{\frac{\theta}{\alpha}}\in K^2_{\frac{\theta}{\alpha}}$ and $\chi_\alpha\in K^2_\alpha$, then by Proposition \ref{kzero} we have the following equations on $L^2$:
\begin{equation}\label{uk2}
\left\{\begin{alignedat}{12}
&\psi\,&-&&\overline{\theta(0)}\alpha\bar\chi_\alpha\,&=A k_0^\theta-\overline{\chi(0)}k_0^\alpha
+\overline{\theta(0)}\overline{\chi_\alpha(0)}\alpha\\
-\theta(0)\bar z&\psi\,&+&&\bar z\alpha \bar\chi_\alpha\,&=A\tilde{k}_0^{\theta}-\bar z\tfrac{\theta}{\alpha}(0)\psi(0)\alpha.
\end{alignedat}
\right.\end{equation}

Taking the scalar product with $k_0^\alpha$ and $\tilde k_0^\alpha$, respectively, we obtain
\begin{equation}\label{uk1}
\left\{\begin{alignedat}{12}
\,{}&\psi(0)&+&\overline{\theta(0)}\alpha(0)&\ \overline{\chi_\alpha(0)}\,&
+&\|k_0^\alpha\|^2&\ \overline{\chi(0)}\,&=\,
\langle A k_0^\theta,k_0^\alpha\rangle\\
\tfrac{\theta}{\alpha}(0)\,&\psi(0)\,&+& &\overline{\chi_\alpha(0)}\,& & & \,&=\,\langle A \tilde k_0^\theta,\tilde k_0^\alpha\rangle
\end{alignedat}
\right.\end{equation}
and, by using the notation $a=\psi(0)$, $b=\overline{\chi_\alpha(0)}$, $c=\overline{\chi(0)}$ in \eqref{uk1}, we obtain the following system of equations for the unknowns $a, b, c$:
\begin{equation}\label{uk4}
\left\{\begin{alignedat}{12}
\,{}&a&+&\overline{\theta(0)}\alpha(0)&\ b\,&
+&\|k_0^\alpha\|^2&\ c\,&=\,
\langle A k_0^\theta,k_0^\alpha\rangle\\
-\tfrac{\theta}{\alpha}(0)\,&a\,&+& &b\,& & & \,&=\,\langle A \tilde k_0^\theta,\tilde k_0^\alpha\rangle.
\end{alignedat}
\right.\end{equation} Hence we proved the following.

\begin{lemma}\label{l5}
Let $A\in\mathcal{T}(\theta,\alpha)\,,\, A=A_{\psi+\bar\chi}$ where $\psi\in K^2_\alpha,\, \chi\in K^2_{\theta}$ and $\chi=\chi_{\frac{\theta}{\alpha}}+\frac{\theta}{\alpha}\chi_\alpha$,
with $\chi_{\frac{\theta}{\alpha}}\in K^2_{\frac{\theta}{\alpha}}$ and $\chi_\alpha\in K^2_\alpha$.
Then the values $\psi(0)$, $\overline{\chi_\alpha(0)}$, $\overline{\chi_{\frac{\theta}{\alpha}}(0)}$ and $\overline{\chi(0)}=\overline{\chi_{\frac{\theta}{\alpha}}(0)}+\overline{\frac{\theta}{\alpha}(0)}\ \overline{\chi_\alpha(0)}$ can be determined from $\langle A k_0^\theta,k_0^\alpha\rangle$, $\langle A \tilde k_0^\theta,\tilde k_0^\alpha\rangle$ as a solution to the system \eqref{uk4}. The solution is unique if  we fix a concrete value of  $\psi(0)$ or $\chi(0)$.
\end{lemma}

Remark that by Corollary \ref{csymbol} we can always choose a concrete value of one of $\psi(0)$ or $\chi(0)$.

From equations \eqref{uk2} we cannot obtain $\chi_{\frac{\theta}{\alpha}}$, hence we need another equality, which we formulate using Proposition \ref{kzero} (5-6) and applying $C_\theta$. We get
$$C_\theta (A^{\alpha,\theta}_{\bar{\psi}}k_0^{\alpha})= \psi(0)\tilde k_0^\theta-\alpha(0)(\tfrac{\theta}{\alpha}\bar z\psi-\tfrac{\theta}{\alpha}(0)\psi(0)\bar z),$$
and by Proposition \ref{rozklad}
$$C_\theta (A^{\alpha,\theta}_\chi k_0^\alpha)=\alpha\bar z\bar\chi_\alpha+\theta\bar z\bar\chi_{\frac{\theta}{\alpha}} -\alpha(0)\tfrac{\theta}{\alpha}\bar z\bar\chi_{\frac{\theta}{\alpha}}.$$
The operator $A$ must satisfy the condition
\begin{equation}\label{uk3}
-\alpha(0)\tfrac{\theta}{\alpha}\bar z\psi+\alpha\bar z\bar\chi_\alpha+(\alpha-\alpha(0))\bar z\tfrac{\theta}{\alpha}\bar\chi_{\frac{\theta}{\alpha}}=C_\theta A^*k_0^\alpha-\psi(0)\bar z\theta.
\end{equation}
To obtain a symbol for the operator $A$, we take the analytic functions $\psi=X$, $\alpha\bar\chi_\alpha=Y$ and $\frac{\theta}{\alpha}\bar\chi_{\frac{\theta}{\alpha}}=Z$  as unknowns, and substitute  $a=\psi(0)$, $b=\overline{\chi_\alpha(0)}$, $c=\overline{\chi(0)}$. Then from \eqref{uk2}, \eqref{uk3} we get the system of equations
\begin{equation}\label{uk5}
\left\{\begin{alignedat}{12}
&\,X\,&-&{}&\ \overline{\theta(0)}\,Y\,& & & &\,=\,&Ak_0^\theta-c\,k_0^\alpha+b\,\overline{\theta(0)}\,\alpha,\\
-\theta(0)&\,X\,&+& & \,Y\,& & & &\,=\,&z\,A\tilde k_0^\theta-\tfrac{\theta}{\alpha}(0)\,a\,\alpha,\\
-\alpha(0)\tfrac{\theta}{\alpha}&\,X\,&+& &\,Y\,&+&(\alpha-\alpha(0))&\,Z &\,=\,&z\,C_\theta A^*k_0^\alpha-a\,\theta.
\end{alignedat}
\right.
\end{equation}
Note that the determinant of the matrix of the coefficients of the system \eqref{uk5} is $\|k_0^\theta\|^2(\alpha-\alpha(0))\neq 0$. Hence the solution of the system \eqref{uk5} is unique.

We have thus the following.

\begin{theorem}
Let $\alpha\leqslant\theta$ be nonconstant inner functions and let $A\in\mathcal{T}(\theta,\alpha)$. Suppose that a certain value is given for either $\psi(0)$ or $\chi(0)$, and let $a$, $b$, $c$ be the corresponding solutions to \eqref{uk4}. If $X, Y, Z$ satisfy the linear system \eqref{uk5},
 then  $A=\ata_{\psi+\bar\chi}$, where $\psi=X$ and $\chi=\frac{\theta}{\alpha}\bar Z+\theta\bar Y$.
\end{theorem}

\begin{proof}

If $A$ has the symbol $\psi+\bar\chi$ with fixed value $\psi(0)$ or $\chi(0)$, then $\psi$ and $\chi$ are uniquely determined. Since, as shown above,
$X=\psi$, $Y=\alpha\bar\chi_\alpha$, $Z=\frac{\theta}{\alpha}\bar\chi_{\frac{\theta}{\alpha}}$ satisfy the system \eqref{uk5}, the result follows from the uniqueness of the solution to that system.
\end{proof}

The characterizations of asymmetric truncated Toeplitz operators in terms of operators of rank two at most, obtained in prevoius sections, allow us also to obtain a symbol for the operator.

In fact, regarding the first characterization, it follows from the proof of Theorem \ref{charakt} that, if $A$ is a bounded operator and satisfies the equality \eqref{chaar}, then $A=\ata_{\psi+\bar\chi}$. Remark that by Corollary \ref{csymbol} we know that $\psi$ and $\chi$ are not unique and we can adjust the value of either $\psi$ or $\chi$ at the origin.

For $\alpha=\theta$ the characterization \eqref{cstar} of truncated Toeplitz operators in Theorem \ref{char-sp} reduces to Sarason's (\cite[Remark, p. 501]{Sarason}). In that case the relation between $\psi,\chi$ in the symbol of $A^\theta_{\psi+\bar\chi}$ and $\mu$, $\nu$ is given by the conjugation $C_\theta$,  namely $\mu=C_\theta\chi$ and $\nu=C_\theta\psi$. Thus  one can also immediately associate a symbol of the form $\psi+\bar\chi$ to  a truncated Toeplitz operator satisfying that equality. In the asymmetric case, however, Theorem \ref{char-sp} unveils a more complex connection between the rank-two operator on the right-hand side of \eqref{cstar} and the symbols of $\ata_{\psi+\bar\chi}$,
and finding a symbol in terms of $\mu$ and $\nu$ for an operator $A$ satisfying equality \eqref{cstar} is more difficult.

 To solve  that problem in the case of asymmetric truncated Toeplitz operators we start with two auxiliary results.

\begin{lemma}\label{ni 1} Let $\psi\in K^2_\alpha\,,\,\chi\in K^2_\theta$.
Assume that
$\chi=\chi_{\frac{\theta}{\alpha}}+\frac{\theta}{\alpha}\chi_{\alpha}$ according to the decomposition $K^2_\theta=K^2_{\frac{\theta}{\alpha}}\oplus \frac{\theta}{\alpha}K^2_\alpha$. If
\begin{equation}\label{6.5a}
\mu=C_\alpha P_\alpha (\tfrac{\bar\theta}{\bar\alpha}\,\chi)+\bar b \tilde k_0^\alpha,\quad\nu=C_\alpha \psi + S^*(\alpha P_{\frac{\theta}{\alpha}}\,\chi)-b\tilde k_0^\theta
\end{equation}
for fixed $b\in \mathbb{C}$, then
\begin{equation}\label{6.5b}
\aligned
\psi&=C_\alpha\nu_\alpha-\big(\overline{\chi_{\frac{\theta}{\alpha}}(0)}-\bar b\, \overline{\tfrac \theta\alpha (0)}\big)k_0^{\alpha},\quad\\
 \chi_\alpha&=C_\alpha\mu-b k_0^\alpha\,,\quad  \chi_{\frac{\theta}{\alpha}}=S_{\frac{\theta}{\alpha}}\nu_{\frac{\theta}{\alpha}}+\big(\chi_{\frac{\theta}{\alpha}}(0) -b\tfrac\theta\alpha(0)\big)k_0^{\frac{\theta}{\alpha}},
\endaligned
\end{equation}
 where $\nu=\nu_\alpha+\alpha\nu_{\frac{\theta}{\alpha}}$ according to the decomposition  $ K^2_\theta=K^2_{\alpha}\oplus{\alpha} K^2_{\frac{\theta}{\alpha}}$.
\end{lemma}

\begin{proof}
Let us calculate using Proposition \ref{proj}
\begin{align*}
P_{\alpha}(S^*(\alpha\chi_{\frac{\theta}{\alpha}}))&=
P_{\alpha}(\bar{z}(\alpha\chi_{\frac{\theta}{\alpha}}-
\alpha(0)\chi_{\frac{\theta}{\alpha}}(0)))=P_\alpha(\bar z\alpha\chi_{\frac{\theta}{\alpha}})\\ &=
\alpha P^-\bar{\alpha}(\bar{z}\alpha\chi_{\frac{\theta}{\alpha}})-
\alpha(0)\chi_{\frac{\theta}{\alpha}}(0)\bar{z}\\
&=\alpha P^-(\bar{z}\chi_{\frac{\theta}{\alpha}})-\alpha(0)\chi_{\frac{\theta}{\alpha}}(0)\bar{z}=(\alpha\chi_{\frac{\theta}{\alpha}}(0)-
\alpha(0)\chi_{\frac{\theta}{\alpha}}(0))\bar{z}\\&=\chi_{\frac{\theta}{\alpha}}(0)\tilde{k}_0^{\alpha}.
\end{align*}
Hence, since $C_\alpha \psi\in K^2_\alpha$ and $\tilde k_0^\theta=\frac{\theta}{\alpha}(0)\tilde k_0^\alpha+\alpha \tilde k_0^{\frac{\theta}{\alpha}}$ (Proposition \ref{cos}), we have $$P_{\alpha}\nu=C_{\alpha}\psi+\big(\chi_{\frac{\theta}{\alpha}}(0)-b\tfrac{\theta}{\alpha}(0)\big)\tilde{k}_0^{\alpha},$$ and thus
\begin{equation}\label{psial}
\psi=C_{\alpha}P_{\alpha}\nu-
C_{\alpha}\big((\chi_{\frac{\theta}{\alpha}}(0)-b\tfrac{\theta}{\alpha}(0))\tilde{k}_0^{\alpha}\big)=
C_\alpha\nu_\alpha-\big(\overline{\chi_{\frac{\theta}{\alpha}}(0)}-\bar b\overline{\tfrac{\theta}{\alpha}(0)}\big)k_0^{\alpha}.\end{equation}

  Now let us consider $\chi$.  Since $\mu=P_{\alpha}C_{\alpha}(\chi \frac{\bar{\theta}}{\bar{\alpha}})+\bar b\tilde k_0^\alpha=C_\alpha\chi_\alpha+\bar b\tilde k_0^\alpha$, then \begin{equation}\label{nial}\chi_\alpha=C_\alpha(\mu-\bar b\tilde k_0^\alpha)=C_\alpha\mu-bk_0^\alpha.\end{equation} On the other hand,  by Proposition \ref{proj}
 \begin{equation}\label{nial1}
 \begin{alignedat}{12} \nu_{\frac{\theta}{\alpha}}=&P_{\frac{\theta}{\alpha}}(\bar{\alpha}\nu)=
P_{\frac{\theta}{\alpha}}(\bar{\alpha}S^*(\alpha\chi_{\frac{\theta}{\alpha}})-b\bar\alpha\tilde k_0^{\theta})=
P_{\frac{\theta}{\alpha}}(\bar{\alpha}\bar{z}(\alpha\chi_{\frac{\theta}{\alpha}}-
\alpha(0)\chi_{\frac{\theta}{\alpha}}(0)))-b\tilde k_0^{\frac{\theta}{\alpha}}\\=&
P_{\frac{\theta}{\alpha}}(\bar{z}\chi_{\frac{\theta}{\alpha}})-
\alpha(0)\chi_{\frac{\theta}{\alpha}}(0)P_{\frac{\theta}{\alpha}}(\bar{\alpha}\bar{z})-b\tilde k_0^{\frac{\theta}{\alpha}}=
P_{\frac{\theta}{\alpha}}(\bar{z}\chi_{\frac{\theta}{\alpha}})-b\tilde k_0^{\frac{\theta}{\alpha}}\\=&\tfrac{\theta}{\alpha}P^-(\tfrac{\bar\theta}{\bar\alpha}\bar z\chi_{\frac{\theta}{\alpha}})-P^-(\bar z\chi_{\frac{\theta}{\alpha}})-b\tilde k_0^{\frac{\theta}{\alpha}}=\bar z\chi_{\frac{\theta}{\alpha}}-\bar z\chi_{\frac{\theta}{\alpha}}(0)-b\tilde k_0^{\frac{\theta}{\alpha}}=S^*_{\frac{\theta}{\alpha}}\chi_{\frac{\theta}{\alpha}}-b\tilde k_0^{\frac{\theta}{\alpha}},\end{alignedat}\end{equation}
since $\tfrac{\bar\theta}{\bar\alpha}\chi_{\frac{\theta}{\alpha}}\perp H^2$.
 Hence, using a proper defect operator and \eqref{D}, we obtain
 \begin{multline}\label{chita}
 \chi_{\frac{\theta}{\alpha}}=(S_{\frac{\theta}{\alpha}}S^*_{\frac{\theta}{\alpha}}+k_0^{\frac{\theta}{\alpha}}\otimes k_0^{\frac{\theta}{\alpha}})\chi_{\frac{\theta}{\alpha}}
 = S_{\frac{\theta}{\alpha}}S^*_{\frac{\theta}{\alpha}}\chi_{\frac{\theta}{\alpha}}+(k_0^{\frac{\theta}{\alpha}}\otimes k_0^{\frac{\theta}{\alpha}}) \chi_{\frac{\theta}{\alpha}}\\
 S_{\frac{\theta}{\alpha}}\nu_{\frac{\theta}{\alpha}}+bS_{\frac{\theta}{\alpha}}\tilde k_0^{\frac{\theta}{\alpha}}+\chi_{\frac{\theta}{\alpha}}(0) k_0^{\frac{\theta}{\alpha}}=\big(\chi_{\frac{\theta}{\alpha}}(0)-b\tfrac{\theta}{\alpha}(0)\big) k_0^{\frac{\theta}{\alpha}}+S_{\frac{\theta}{\alpha}}\nu_{\frac{\theta}{\alpha}}.
 \end{multline}
\end{proof}

\begin{lemma}\label{ni}
Let $A\in\mathcal{T}(\theta,\alpha)$ satisfy the equation
\begin{equation*}A-S_\alpha^* A S_\theta=\mu\otimes \tilde{k}_0^\theta+\tilde{k}_0^\alpha\otimes \nu\end{equation*} for $\mu\in K^2_\alpha$, $\nu\in K^2_\theta$. Then $\mu$ and $\nu$ can be chosen such that $P_{\frac{\theta}{\alpha}}(\bar\alpha\nu)$ is orthogonal to $\tilde k_0^{\frac{\theta}{\alpha}}$. In this case, $\mu$ and $\nu$ are uniquely determined.
\end{lemma}
\begin{proof}
According to the decomposition $K^2_\theta=K^2_\alpha\oplus\alpha K^2_{\frac{\theta}{\alpha}}$ we can write $\nu=P_\alpha\nu+\alpha P_{\frac{\theta}{\alpha}}\bar\alpha\nu=\nu_\alpha+\alpha \nu_{\frac{\theta}{\alpha}}$.
 By Proposition \ref{cos} recall that $\tilde k_0^\theta=\frac{\theta}{\alpha}(0)\tilde k_0^\alpha+\alpha\tilde k_0^{\frac{\theta}{\alpha}}$.
Let us define $$\mu'=\mu+\|\tilde k_0^{\frac{\theta}{\alpha}}\|^{-2}\overline{\langle \nu_{\frac{\theta}{\alpha}},\tilde k_0^{\frac{\theta}{\alpha}}\rangle} \tilde k_0^\alpha,\quad
\nu'=\nu-\|\tilde k_0^{\frac{\theta}{\alpha}}\|^{-2}\langle \nu_{\frac{\theta}{\alpha}},\tilde k_0^{\frac{\theta}{\alpha}}\rangle \tilde k_0^\theta$$ (see \eqref{6.2a}). We have

\begin{equation*}A-S_\alpha^* A S_\theta=\mu'\otimes \tilde{k}_0^\theta+\tilde{k}_0^\alpha\otimes \nu',\qquad \mu'\in K^2_\alpha,\quad \nu'\in K^2_\theta.\end{equation*}
Moreover, $\nu'_{\frac{\theta}{\alpha}}$ is orthogonal to $\tilde k_0^{\frac{\theta}{\alpha}}$, since
$$\langle \nu'_{\frac{\theta}{\alpha}},\tilde k_0^{\frac\theta\alpha} \rangle= \langle \nu_{\frac{\theta}{\alpha}}-\|\tilde k_0^{\frac{\theta}{\alpha}}\|^{-2}\langle \nu_{\frac{\theta}{\alpha}},\tilde k_0^{\frac{\theta}{\alpha}}\rangle \tilde k_0^{\frac{\theta}{\alpha}},\tilde k_0^{\frac{\theta}{\alpha}} \rangle=0.$$
To show uniqueness it is enough to note that
$\mu\otimes \tilde k_0^\theta+\tilde k_0^\alpha\otimes \nu=0$ if and only if $\mu=0$ and $\nu=0$. Assume then that
$\mu\otimes \tilde k_0^\theta+\tilde k_0^\alpha\otimes \nu=0$, which is equivalent to
 $\mu=c_1\tilde k_0^\alpha$, $\nu=c_2\tilde k_0^\theta$ and $c_1+\bar c_2=0$ for $c_1, c_2\in\mathbb{C}$. Then by Proposition \ref{cos} (4) $$P_{\frac{\theta}{\alpha}}\bar \alpha\nu=c_2P_{\frac{\theta}{\alpha}}\bar\alpha \tilde k_0^\theta=c_2\tilde k_0^{\frac{\theta}{\alpha}}.$$
Since $P_{\frac{\theta}{\alpha}}\bar \alpha\nu\perp\tilde k_0^{\frac{\theta}{\alpha}}$, we obtain that $c_2=0$, which implies $c_1=0$. Hence $\mu=0$ and $\nu=0$.
\end{proof}


When investigating  symbols of the asymmetric truncated Toeplitz operator, it is worth to have in mind
 Corollary \ref{csymbol} saying that it is enough to find one of them. 

\begin{theorem}\label{pr1}
Let $\alpha, \theta$ be nonconstant inner functions such that $\alpha \leqslant \theta$ and let $A$ be a bounded operator satisfying
\begin{equation}\label{cstar1}A-S_\alpha^* A S_\theta=\mu\otimes \tilde{k}_0^\theta+\tilde{k}_0^\alpha\otimes \nu\end{equation}
for $\mu\in K^2_\alpha$, $\nu\in K^2_\theta$. Then $A=\ata_{\psi+\bar\chi}$, where
\begin{align*}
\psi=&C_\alpha P_\alpha(\nu- c \tilde k_0^\theta)= C_\alpha P_\alpha\nu-\bar c\,\overline{ \tfrac{\theta}{\alpha}(0)} k_0^\alpha\in K^2_\alpha\quad\text{and} \\ \chi=&S_{\frac{\theta}{\alpha}} P_{\frac{\theta}{\alpha}}\bar\alpha(\nu- c\, \tilde k_0^\theta)+\tfrac{\theta}{\alpha}C_\alpha(\mu+\bar c\tilde k_0^\alpha)\\ =&
(S_{\frac{\theta}{\alpha}} P_{\frac{\theta}{\alpha}}\bar\alpha\nu+ c\,\tfrac{ \theta}{\alpha}(0) k_0^{\frac{\theta}{\alpha}})+\tfrac{\theta}{\alpha}(C_\alpha\mu+ c\, k_0^\alpha)\in K^2_\theta=K^2_{\frac{\theta}{\alpha}}\oplus \tfrac{\theta}{\alpha}K^2_\alpha\quad\text{with}
\\ c=&\langle P_{\frac{\theta}{\alpha}}\bar\alpha\nu,\tilde k_0^{\frac{\theta}{\alpha}}\rangle \|\tilde k_0^{\frac{\theta}{\alpha}}\|^{-2}.
\end{align*}
\end{theorem}

\begin{proof}
Let $\mu_0=\mu+\bar c\,\tilde k_0^\alpha$, $\nu_0=\nu- c \,\tilde k_0^\theta$ with $c=\langle P_{\frac{\theta}{\alpha}}\bar\alpha\nu,\tilde k_0^{\frac{\theta}{\alpha}}\rangle \|\tilde k_0^{\frac{\theta}{\alpha}}\|^{-2}.$ We have $P_{\frac{\theta}{\alpha}}(\bar\alpha\nu_0)\bot\tilde k_0^{\frac{\theta}{\alpha}}$,  and \eqref{cstar1} is also satisfied, i.e.,
\begin{equation}\label{6.6a}A-S_\alpha^* A S_\theta=\mu_0\otimes \tilde{k}_0^\theta+\tilde{k}_0^\alpha\otimes \nu_0.\end{equation}
On the other hand,  $A\in \mathcal{T}(\theta,\alpha)$ so $A=A^{\theta,\alpha}_{\psi+\bar\chi}$ with $\psi\in K^2_\alpha$ and $\chi\in K^2_\theta$ and, by the first part of the proof of Theorem \ref{char-sp}, especially \eqref{symbol}, and \eqref{6.2a}, there are also $\mu'=C_{\alpha}P_{\alpha} (\chi\tfrac{\bar{\theta}}{\bar{\alpha}})+\bar b\tilde k_0^\alpha\in K^2_\alpha$ and $\nu'=C_{\alpha} \psi+S^*_\theta (\alpha P_{\frac{\theta}{\alpha}}\chi)-b\tilde k_0^\theta\in K^2_\theta$ ($b\in\mathbb{C}$) such that
\begin{equation}\label{6.6b}A-S_\alpha^* A S_\theta=\mu'\otimes \tilde{k}_0^\theta+\tilde{k}_0^\alpha\otimes \nu'.\end{equation}
Take
\[b=-{{\chi_{\frac{\theta}{\alpha}}(0)}\  \overline{\tfrac{\theta}{\alpha}(0)}}\ {\|\tilde k_0^{\frac{\theta}{\alpha}}\|^{-2}}\,,\]
where $\chi_{\frac{\theta}{\alpha}}=P_{\frac{\theta}{\alpha}}\chi$. Using calculations as in \eqref{nial1},
we have $P_{\frac{\theta}{\alpha}}(\bar\alpha(\nu'- b\,\tilde k_0^\theta))\,\bot\,\tilde k_0^{\frac{\theta}{\alpha}}$, because
\begin{multline*}
\langle P_{\frac{\theta}{\alpha}}\bar\alpha(\nu'- b\,\tilde k_0^\theta),\tilde k^{\frac{\theta}{\alpha}}_0\rangle=\langle S^*\chi_{\frac{\theta}{\alpha}},\tilde k^{\frac{\theta}{\alpha}}_0\rangle- b\langle P_{\frac{\theta}{\alpha}}\bar\alpha \tilde k_0^{\theta},\tilde k_0^{\frac{\theta}{\alpha}}\rangle\\=\langle \bar z(\chi_{\frac{\theta}{\alpha}}-\chi_{\frac{\theta}{\alpha}}(0)),\bar z(\tfrac{\theta}{\alpha}-\tfrac{\theta}{\alpha}(0))\rangle- b\langle \tilde k_0^{\frac{\theta}{\alpha}},\tilde k_0^{\frac{\theta}{\alpha}}\rangle=-\chi_{\frac{\theta}{\alpha}}(0)\overline{\tfrac{\theta}{\alpha}(0)}- b\|\tilde k_0^{\frac{\theta}{\alpha}}\|^2=0.\end{multline*}
Hence, by Lemma \ref{ni}, $\mu_0=\mu'$ and $\nu_0=\nu'$. On the other hand, by Lemma \ref{ni 1}
\begin{equation*}\psi=C_\alpha P_\alpha\nu_0-\big(\overline{\chi_{\frac{\theta}{\alpha}}(0)}+\overline{\chi_{\frac{\theta}{\alpha}}(0)}|\tfrac{\theta}{\alpha}(0)|^2 \|\tilde k_0^{\frac{\theta}{\alpha}}\|^{-2}\big)k_0^\alpha=C_\alpha P_\alpha\nu_0-
\overline{\chi_{\frac{\theta}{\alpha}}(0)} \|\tilde k_0^{\frac{\theta}{\alpha}}\|^{-2}k_0^\alpha
\end{equation*}
 and
\begin{multline*}
\chi=S_{\frac{\theta}{\alpha}} P_{\frac{\theta}{\alpha}}(\bar\alpha\nu_0)+\chi_{\frac{\theta}{\alpha}}(0)\big(1+\tfrac{\big| \tfrac{\theta}{\alpha}(0)\big|^2}{\|\tilde k_0^{\frac{\theta}{\alpha}}\|^{2}}\big)k_0^{{\frac{\theta}{\alpha}}}+{\tfrac{\theta}{\alpha}}\big(C_\alpha\mu_0+
\chi_{\frac{\theta}{\alpha}}(0)\tfrac{\overline{\tfrac\theta\alpha(0)}}{\|\tilde k_0^{\frac{\theta}{\alpha}}\|^{2}} k_0^\alpha\big)
\\=S_{\frac{\theta}{\alpha}} P_{\frac{\theta}{\alpha}}(\bar\alpha\nu_0)+\chi_{\frac{\theta}{\alpha}}(0)\|\tilde k_0^{\frac{\theta}{\alpha}}\|^{-2}k_0^{{\frac{\theta}{\alpha}}}+{\tfrac{\theta}{\alpha}}\big(C_\alpha\mu_0+\chi_{\frac{\theta}{\alpha}}(0)\overline{\tfrac{\theta}{\alpha}(0)}\|\tilde k_0^{\frac{\theta}{\alpha}}\|^{-2} k_0^\alpha\big)\\=S_{\frac{\theta}{\alpha}} P_{\frac{\theta}{\alpha}}(\bar\alpha\nu_0)+{\tfrac{\theta}{\alpha}}C_\alpha\mu_0+{{\chi_{\frac{\theta}{\alpha}}(0)}}\,{\|\tilde k_0^{\frac{\theta}{\alpha}}\|^{-2}}\,k_0^\theta.
\end{multline*}
Therefore, taking Corollary \ref{csymbol} into account, we conclude that $A=A^{\theta,\alpha}_{\psi'+\overline{\chi'}}$ with
$\psi'=C_\alpha P_\alpha\nu_0=C_\alpha P_\alpha(\nu- c\tilde k_0^\theta)$ and $\chi'=S_{\frac{\theta}{\alpha}} P_{\frac{\theta}{\alpha}}(\bar\alpha\nu_0)+{\frac{\theta}{\alpha}}C_\alpha\mu_0=S_{\frac{\theta}{\alpha}} P_{\frac{\theta}{\alpha}}(\bar\alpha(\nu- c\tilde k_0^\theta))+{\frac{\theta}{\alpha}}C_\alpha(\mu+\bar c\tilde k_0^\alpha)$.
\end{proof}

{\bf Acknowledgments.}{The work of the first author was partially supported by Funda\c{c}\~{a}o para a Ci\^{e}ncia e a Tecnologia (FCT/Portugal), through Project PEst-OE/EEI/LA0009/2013. The research of the third and the forth authors was financed by the Ministry of Science and Higher Education of the Republic of Poland.}

\bibliographystyle{amsplain}

\end{document}